\documentclass[10pt]{article}
\usepackage{amsmath,amsthm,amssymb}
\usepackage{enumitem}

\usepackage[margin=2.3 cm,nohead]{geometry}
\usepackage[numbers,sort&compress]{natbib}
\usepackage{hyperref}
\usepackage{graphicx}
\usepackage{booktabs}    
\providecommand{\doi}[1]{\url{https://doi.org/#1}}

\usepackage{xcolor}
\newif\ifshowrev
\showrevfalse
\newcommand{\rev}[1]{\ifshowrev\textcolor{blue}{#1}\else#1\fi}
\newenvironment{revblock}{\ifshowrev\color{blue}\fi}{}

\setlist[enumerate]{label=(\roman*), align=left}

\newtheorem{theorem}{Theorem}
\newtheorem{lemma}[theorem]{Lemma}
\newtheorem{definition}{Definition}
\newtheorem{corollary}[theorem]{Corollary}
\newtheorem{proposition}[theorem]{Proposition}
\newtheorem{remark}{Remark}

\newtheorem{assumption}{Assumption}

\begin{document}
\title{Adaptive Metrics for Norm-Minimization-Based Outer Approximation in Convex Vector Optimization}

\author{
Mohammed Alshahrani \thanks{Department of Mathematics, King Fahd University of Petroleum \& Minerals, Dhahran, 31261, Saudi Arabia\\
Interdisciplinary Research Center for Smart Mobility and Logistics, King Fahd University of Petroleum \& Minerals, Dhahran, 31261, Saudi Arabia, (e-mail:{\tt mshahrani@kfupm.edu.sa}).}
}

\maketitle

\begin{abstract}
We develop an adaptive-metric framework for norm-minimization-based outer approximation algorithms in bounded convex vector optimization. The key idea is to let the scalarization metric vary across iterations while measuring approximation error in a fixed Euclidean norm. This enables the algorithm to exploit problem geometry dynamically. Our approach rests on two theoretical foundations. First, we prove that the improved Euclidean convergence rate $O(k^{2/(1-q)})$---previously known only for the standard $\ell_2$-norm---extends to all fixed inner-product norms. \rev{Second, our main result proves that, under the stated assumptions, the adaptive-metric algorithm attains the same rate $O(k^{2/(1-q)})$ without strict-convexity, smoothness, or dispersion hypotheses. The proof uses uniform conditioning from Tikhonov regularization to adapt the Euclidean packing argument to the moving metric. Conditioning affects only the multiplicative constant.} Numerical experiments on three test problems \rev{yield fitted slopes consistent with} the theoretical convergence rate\rev{ for both the fixed Euclidean and adaptive metrics, with comparable evaluation counts and no observed speed advantage on these tests}. Our results provide a rigorous foundation for adaptive metric selection in convex vector optimization.
\end{abstract}


\medskip



\section{Introduction}

Vector optimization provides a unifying framework for multiobjective decision-making under partial orders induced by convex cones rather than simple componentwise dominance, for more details, see \cite{ansari_vector_2018}. In the convex setting, one is interested in approximating the \emph{upper image}
\[
P := \Gamma(X) + C,
\]
where $C\subset\mathbb{R}^q$ is a closed, solid, pointed cone, $X\subset\mathbb{R}^n$ is a compact convex feasible set, and $\Gamma:X\to\mathbb{R}^q$ is continuous and $C$--convex. Efficient solution methods for such convex vector optimization problems (CVOPs) are central in applications from finance and risk to control and machine learning; see, for instance, \cite{ararat_norm_2022,keskin_outer_2023} and the references therein.

Among objective-space methods, \emph{outer approximation} algorithms have emerged as particularly interesting. Originating from Benson's algorithm for linear multiobjective optimization and its generalizations, these methods iteratively refine a polyhedral outer approximation of $P$ by solving scalarization problems and adding supporting halfspaces, until the approximation is sufficiently accurate in a suitable metric. For CVOPs that are \emph{bounded}---in the sense that the upper image satisfies $P\subseteq\{y\}+C$ for some $y\in\mathbb{R}^q$---and have polyhedral ordering cones, several such algorithms have been developed based on Pascoletti--Serafini scalarizations, norm-minimizing scalarizations, and various vertex and direction selection rules; see, for instance, \cite{ararat_norm_2022,keskin_outer_2023}.

A recent line of work by Ararat et al. \cite{ararat_norm_2022,ararat_convergence_2024} introduced and analyzed an outer approximation algorithm driven by a \emph{norm-minimizing} scalarization. At each iteration, a reference point $v$ (typically a vertex of the current outer approximation) is chosen and a scalar convex program is solved that computes the distance from $v$ to a compact slice $A$ of the upper image with respect to a fixed norm (a slice-constrained variant of the norm-minimizing scalarization, following \cite[Algorithm~1]{ararat_convergence_2024}). The resulting primal--dual optimal pair yields, respectively, a new weak minimizer of the original CVOP and a supporting halfspace of that slice. This algorithm has two key theoretical properties: (i) it is the first CVOP algorithm for which \emph{finiteness} is proved under mild assumptions, and (ii) it admits a nontrivial \emph{convergence-rate analysis} based on the Hausdorff distance.

In \cite{ararat_convergence_2024}, the authors establish that, for an arbitrary norm on the objective space, the outer approximation $A_k$ converges to the compact slice $A:=P\cap S(\gamma)$ at the Hausdorff rate $O(k^{1/(1-q)})$ after $k$ iterations (equivalently, the cone extension $A_k+C$ approximates the full upper image $P$ at the same rate, by \cite[Cor.~6.10]{ararat_convergence_2024}), where $q$ is the dimension of the objective space and $q\ge 2$. This result is obtained by relating the sequence of approximating polytopes to an $H(r,A)$--sequence of cutting, see Definition \ref{def:H_sequence}, in the sense of Kamenev and Lotov et al.\ \cite{kamenev1992class,lotov2004interactive}, and then invoking general approximation results for convex bodies by polyhedra. In addition, \cite{ararat_convergence_2024} contains a refined analysis, in case of Euclidean-norm-based scalarizations, by exploiting the special geometry of the $\ell_2$-norm, the authors prove an improved convergence rate
\[
\delta_H(A_k,A) = O\bigl(k^{2/(1-q)}\bigr),
\]
measured against that same compact slice $A$ (Section~\ref{sec:preliminaries}), which is known to be the best possible in the classical multiobjective setting \cite{lotov2004interactive}. To the best of our knowledge, this is the first convergence-rate result of this sharp order for a general CVOP algorithm.

The improved convergence rate based on the Euclidean norm raises the following natural question. To what extent is this enhancement tied to the standard Euclidean structure, and to what extent does it reflect more general Hilbertian geometry? While the generic convergence rate of order $O(k^{1/(1-q)})$ in \cite{ararat_convergence_2024} holds for arbitrary norms, the proof of the rate of order $O(k^{2/(1-q)})$ uses several arguments that appear to be specific to the $\ell_2$-norm (e.g., Pythagorean identities and Euclidean packing estimates). It is not immediately clear whether similar arguments can be carried out for other norms, nor how the geometry of the norm (e.g., asphericity, uniform convexity) influences the achievable exponent.

The aim of this paper is to close this gap at least for the important class of \emph{inner-product norms} on $\mathbb{R}^q$. Any such norm can be written in the form
\[
\|y\|_M := \sqrt{y^\top M y},\qquad y\in\mathbb{R}^q,
\]
for some symmetric positive definite matrix $M$. The key observation is that the linear isomorphism $T:=M^{1/2}$ is an isometry between $(\mathbb{R}^q,\|\cdot\|_M)$ and the Euclidean space $(\mathbb{R}^q,\|\cdot\|_2)$, and that the norm-minimization--based algorithm of \cite{ararat_convergence_2024} is invariant under such transformations, provided one simultaneously transforms the ordering cone, the upper image and the compact slice in the objective space; we make this precise in Section~\ref{sec:inner_product_invariance} (see Proposition~\ref{prop:algorithm-equivalence}).


The approximation of Pareto frontiers and upper images in multiobjective optimization has a rich history spanning several decades. We briefly survey the main threads of this literature, focusing on outer approximation methods, scalarization techniques, convergence rate theory, and alternative approximation paradigms.

The systematic study of objective-space methods for multiobjective optimization began with Benson's seminal outer approximation algorithm for multiple objective linear programming~\cite{benson_outer_1998}. Benson's algorithm iteratively refines a polyhedral outer approximation of the efficient set by solving weighted-sum scalarizations and adding cutting hyperplanes, providing the first practical method capable of generating all efficient extreme points in the outcome set. The extension of Benson's approach to convex multiobjective problems was achieved by Ehrgott et al.~\cite{ehrgott_approximation_2011}, who proved weak $\varepsilon$-nondominance guarantees and gave a construction of separating hyperplanes in the nonsmooth setting. L\"ohne et al.~\cite{lohne_primal_2014} introduced primal and dual approximation algorithms that operate simultaneously on the upper and lower images. Their framework allows solid pointed polyhedral ordering cones and provides an $\varepsilon$-solution concept with rigorous approximation guarantees. The comprehensive treatment of vector optimization from a set-optimization perspective, including the theoretical foundations for infimum and supremum concepts, can be found in L\"ohne's monograph~\cite{lohne_vector_2011} and the references therein. Recent work by Keskin and Ulus~\cite{keskin_outer_2023} provides a systematic comparison of vertex selection rules and direction choices within the Pascoletti--Serafini framework, while Wagner et al.~\cite{wagner2023algorithms} extend outer approximation methods to unbounded problems. For general background on vector and multicriteria optimization, we refer to Jahn~\cite{jahn_vector_2011}. Moreover, extensions to nonconvex settings have been developed using branch-and-bound techniques~\cite{niebling2019branch,nobakhtian2017benson}.

The choice of scalarization has significant implications for both the theoretical properties and computational efficiency of outer approximation algorithms. The Pascoletti--Serafini scalarization, which finds the minimum distance from a reference point to the upper image along a prescribed direction, has been widely adopted due to its flexibility and geometric interpretability; see~\cite{keskin_outer_2023} for an extensive computational study comparing different parameter choices. However, this scalarization introduces direction-biasedness: the quality of the approximation can depend strongly on the choice of search directions. To address this limitation, Ararat et al.~\cite{ararat_norm_2022} introduced a norm-minimizing scalarization that computes the distance from a reference point to the upper image with respect to a fixed norm, without requiring a direction parameter. This approach yields an algorithm that is direction-unbiased in the sense of Klamroth et al.~\cite{klamroth2003unbiased}, who first identified direction-biasedness as a potential source of inefficiency in multiobjective approximation schemes. The norm-minimizing approach also enabled the first finiteness proof for a convex vector optimization algorithm under mild assumptions~\cite{ararat_norm_2022}. The interplay between scalarization choice and computational tractability in constrained settings is further explored in~\cite{klamroth2007constrained}.

The study of convergence rates for polyhedral approximation of convex bodies has deep roots in computational geometry and convex analysis. Kamenev~\cite{kamenev1992class,kamenev2002conjugate} developed a general theory of adaptive algorithms for approximating convex bodies by polyhedra, introducing the concept of $H(r,A)$-sequences---sequences of cutting halfspaces that remove at least a fraction $r$ of the current approximation error at each step. For such sequences, the Hausdorff approximation error decreases at the rate $O(k^{1/(1-q)})$. The improved rate $O(k^{2/(1-q)})$ for Euclidean approximation is a classical result in convex geometry, originating with Gruber and Kenderov~\cite{gruber_approximation_1982}, who proved that it is the best possible rate for smooth convex bodies. Lotov et al.~\cite{lotov2004interactive} provided a comprehensive treatment of these ideas in the context of interactive decision maps and Pareto frontier visualization; we refer to this monograph and the references therein for the extensive classical literature on polyhedral approximation of convex bodies. The improvement in the Euclidean case stems from packing estimates for spherical caps. In Euclidean geometry, the number of nearly disjoint caps of radius $\varepsilon$ on the unit sphere grows as $\varepsilon^{1-q}$. For general norms, the corresponding exponent may be worse, which limits the achievable approximation rate. Ararat et al.~\cite{ararat_convergence_2024} were the first to establish these convergence rates for a specific CVOP algorithm, proving that their norm-minimizing algorithm generates an $H(r,A)$-sequence and thereby inherits the corresponding rate guarantees; see also the references therein for connections to approximation theory and convex geometry. Their analysis leaves open the question of whether the improved Euclidean exponent can be achieved for other norms---a question we resolve affirmatively for inner-product norms in the present work.

While polyhedral outer approximation has been the dominant approach, alternative paradigms have emerged. Eichfelder and Warnow~\cite{eichfelder_approximation_2022} introduced a box-coverage method that approximates the nondominated set using a finite collection of boxes respecting the natural ordering, providing a simpler structure than polyhedra for subsequent analysis. The concept of local upper bounds, which plays a key role in branch-and-bound methods for nonconvex multiobjective optimization, has recently been extended to general polyhedral ordering cones by Eichfelder and Ulus~\cite{eichfelder_local_2026}. Mixed-integer extensions of convex multiobjective optimization have been developed by De~Santis et al.~\cite{desantis2020solving}, combining outer approximation with integer programming techniques.

The present paper extends the convergence rate theory of~\cite{ararat_convergence_2024} in two directions. First, we show that the improved convergence rate based on the Euclidean norm is not special to the $\ell_2$-norm but extends to all inner-product norms via a linear isometry argument. Second, we introduce and analyze adaptive inner-product metrics that vary across iterations. \rev{Under the assumptions of Theorem~\ref{thm:rate}, the adaptive algorithm attains the same exponent $2/(1-q)$ as fixed inner-product norms (Theorem~\ref{thm:inner-product-rate}), without strict-convexity, smoothness, or dispersion hypotheses. This is the first rigorous convergence-rate guarantee for adaptive metric selection in norm-minimization-based vector optimization.}


The contributions of this paper are fivefold.
\begin{enumerate}
    \item (Section~\ref{sec:inner_product_invariance}, Theorem~\ref{thm:inner-product-rate}) We show that the improved Hausdorff convergence rate $O(k^{2/(1-q)})$, previously established only for Euclidean norm-based scalarizations, extends to all inner-product norms on the objective space. By exploiting the linear isometric structure induced by positive definite matrices, we prove that the norm-minimization--based outer approximation algorithm is invariant under appropriate coordinate transformations, allowing the Euclidean convergence analysis to be transferred verbatim to arbitrary inner-product geometries.

    \item (Section~\ref{sec:inner_product_invariance}, Theorem~\ref{thm:inner-product-rate}) We derive explicit convergence bounds in the original coordinates, in which the constants depend on the spectral properties---the extreme eigenvalues $\lambda_{\min}(M),\lambda_{\max}(M)$, equivalently the condition number $\kappa(M)$---of the underlying inner product and on geometric characteristics of the transformed slice $\tilde A=TA$ (through its circumradius). This clarifies the precise role played by the choice of norm in the convergence behavior of the algorithm.

    \item (Section~\ref{sec:convergence}, Theorems~\ref{thm:baseline_correctness_slice} and~\ref{thm:euclid_convergence_adaptive}) We analyze an adaptive-metric variant in which the inner product used in the scalarization varies across iterations while the approximation error is evaluated in a fixed Euclidean norm. Under uniform spectral bounds on the adaptive metrics, we establish correctness and convergence to the slice in the fixed Euclidean norm, together with computable metric-quality and selection-proxy bounds relating the adaptive scalarization to the Euclidean error (the convergence-rate constant is quantified by the rate theorem, Theorem~\ref{thm:rate}).

    \item \rev{(Section~\ref{sec:convergence}, Theorem~\ref{thm:rate}) Our main result proves that, under the theorem's assumptions, the adaptive-metric algorithm attains the improved Hausdorff rate $O(k^{2/(1-q)})$, with the same sharp exponent as fixed inner-product norms and without strict-convexity, smoothness, or dispersion hypotheses on the image set. The proof uses uniform conditioning from the $\varepsilon_0$-regularization to adapt the packing argument of~\cite{ararat_convergence_2024} to the moving metric. Conditioning affects only the multiplicative constant. Whether this exponent persists for an unregularized moving metric with unbounded condition number remains open.}

    \item (Section~\ref{sec:experiments}) We present numerical experiments on three test problems from the literature. The experiments \rev{are assessed against} the theoretical rate $O(k^{2/(1-q)})$ for \rev{the fixed Euclidean and adaptive metrics (Theorems~\ref{thm:inner-product-rate} and~\ref{thm:rate}), with slopes consistent with the predicted exponents, comparable evaluation counts, and no observed speed advantage on these tests}. \rev{In the four main $q\le 3$ runs, the cut-normal second-moment matrix satisfies $\lambda_{\min}(\Sigma_k)>0$ at termination; no uniform-in-$k$ floor is claimed (Remark~\ref{rem:dispersion}).} For a high-dimensional case ($q = 4$), we compare vertex-finding strategies as heuristics, without establishing a convergence guarantee.
\end{enumerate}

Together, these results provide a rigorous foundation for the use of geometry-aware and adaptive inner-product metrics in norm-minimization--based convex vector optimization.
    
The remainder of the paper is organized as follows. Section~\ref{sec:preliminaries} provides the notation and problem setting. In Section~\ref{sec:inner_product_invariance}, we extend the improved convergence rate from the Euclidean norm to all inner-product norms. The adaptive-metric framework, in which the scalarization norm varies across iterations, is introduced in Section~\ref{sec:adaptive_framework}. Section~\ref{sec:convergence} develops the convergence analysis of this framework and \rev{establishes its improved convergence rate}. Numerical experiments on three test problems are reported in Section~\ref{sec:experiments}. Finally, Section~\ref{sec:conclusion} concludes with open questions.


\section{Preliminaries and Problem Setting}\label{sec:preliminaries}

Let $q\in\mathbb{N}$ and let $\|\cdot\|$ be any norm on $\mathbb{R}^q$ with dual norm $\|\cdot\|_*$. For a set $A\subset\mathbb{R}^q$, we denote by $\operatorname{int} A$, $\operatorname{cl} A$, $\operatorname{bd} A$, $\operatorname{conv} A$, $\operatorname{cone} A$ and $\operatorname{ext} A$ its interior, closure, boundary, convex hull, conic hull, and set of extreme points (vertices), respectively. We further write $\mathbb{S}^{q-1} := \{u\in\mathbb{R}^q : \|u\|_2 = 1\}$ for the Euclidean unit sphere, and $I_q$ (or simply $I$, when the dimension is clear) for the $q\times q$ identity matrix. Given nonempty sets $A,B\subset\mathbb{R}^q$, the \emph{Hausdorff distance} of $A$ and $B$ with respect to $\|\cdot\|$, denoted $\delta_H(A,B;\|\cdot\|)$, is
\[
\delta_H(A,B;\|\cdot\|) := \max\left\{ \sup_{y\in A} d(y,B),\ \sup_{z\in B} d(z,A) \right\},
\quad
d(y,B) := \inf_{z\in B} \|y-z\|.
\]
When the underlying norm is clear from the context we write simply $\delta_H(A,B)$; in particular, $\delta_H(\cdot,\cdot;\|\cdot\|_2)$ and $\delta_H(\cdot,\cdot;\|\cdot\|_M)$ denote the Hausdorff distance with respect to the Euclidean norm $\|\cdot\|_2$ and an inner-product norm $\|\cdot\|_M$ (both introduced below), respectively.

For a nonempty convex set $A\subset\mathbb{R}^q$ and $w\in\mathbb{R}^q\setminus\{0\}$, the \emph{support function} of $A$ in direction $w$ is
\[
h_A(w) := \inf_{y\in A} w^\top y.
\]
The \emph{supporting halfspace} of $A$ with outer normal $w$ is
\[
\mathcal{H}(w,A) := \{y\in\mathbb{R}^q : w^\top y \ge h_A(w)\}.
\]
If $\bar y\in A$ attains the infimum, i.e., $w^\top \bar y = h_A(w)$, then $\mathcal{H}(w,A)$ is called a \emph{supporting halfspace of $A$ at $\bar y$}, and $\operatorname{bd}\mathcal{H}(w,A) = \{y : w^\top y = w^\top \bar y\}$ is the corresponding \emph{supporting hyperplane}. Let $C\subset \mathbb{R}^q$ be a closed convex cone. Its dual cone is
\[
C^+ := \{ w\in\mathbb{R}^q \mid w^\top y \ge 0\ \forall\,y\in C\}.
\]
For a nonempty closed convex set $A\subset\mathbb{R}^q$, the \emph{recession cone} of $A$ is
\[
\mathrm{rec}(A) := \{ d\in\mathbb{R}^q \mid y + td \in A\ \forall\,y\in A,\ t\ge 0\}.
\]
For the ordering cone $C$, the \emph{efficient boundary} (or \emph{efficient frontier}) of a set $A\subset\mathbb{R}^q$ is
\[
\partial_{\mathrm{eff}} A := \{ y\in A \mid (\{y\}-\operatorname{int} C)\cap A = \emptyset \},
\]
i.e., the set of weakly $C$-minimal points of $A$ (equivalently, the part of $\operatorname{bd} A$ not dominated by other points of $A$).
For symmetric matrices $A, B \in \mathbb{R}^{q \times q}$, we write $A \preceq B$ (equivalently, $B \succeq A$) to denote the \emph{Loewner order}.\footnote{The \emph{Loewner order} on the space of symmetric $q\times q$ matrices is defined by $A\preceq B \iff B-A$ is positive semidefinite (equivalently, all eigenvalues of $B-A$ are nonnegative).} We write $A \prec B$ (equivalently, $B \succ A$) if $B - A$ is positive definite. In particular, $M \succ 0$ means $M$ is positive definite.

\begin{assumption}[Standing hypotheses]\label{ass:standing}
Throughout this paper, we assume:
\begin{enumerate}
\item[(a)] $C\subset\mathbb{R}^q$ is a closed, solid, pointed, nontrivial, polyhedral cone.
\item[(b)] $X\subset\mathbb{R}^n$ is a nonempty compact convex set with $\operatorname{int} X\neq\emptyset$.
\item[(c)] $\Gamma:X\to\mathbb{R}^q$ is continuous and $C$--convex, i.e.,
\[
\Gamma(\lambda x_1 + (1-\lambda)x_2) \le_C \lambda \Gamma(x_1) + (1-\lambda)\Gamma(x_2)
\quad
\forall\,x_1,x_2\in X,\ \lambda\in[0,1],
\]
where the partial order induced by $C$ is given by $y \le_C z \iff z-y \in C$.
\end{enumerate}
\end{assumption}

We consider the convex vector optimization problem
\begin{equation}\label{P}
\tag{P}
\min \Gamma(x) \ \text{w.r.t. }\le_C \quad\text{s.t. } x\in X.
\end{equation}
The \emph{upper image} of \eqref{P} is
\[
P := \Gamma(X) + C.
\]
Under Assumption~\ref{ass:standing}, $P$ is a closed convex set with $P = P+C$, and \eqref{P} is bounded in the sense that $P\subset \{y\} + C$ for some $y\in\mathbb{R}^q$. To facilitate the subsequent outer approximation procedure, we introduce a compact slice of the upper image by intersecting it with a suitably chosen halfspace $S(\gamma)$ (see~\cite[Section~4]{ararat_convergence_2024} for details). To this end, fix $\bar w\in\operatorname{int} C^+$ and $\gamma\in\mathbb{R}$ such that
\[
\Gamma(X) \subset \operatorname{int} S(\gamma),
\qquad
S(\gamma) := \{ y\in\mathbb{R}^q \mid \bar w^\top y \le \gamma\}.
\]
Then $A := P \cap S(\gamma)$ is a nonempty convex compact set and $P = A + C$ (see~\cite[Remark~4.4]{ararat_convergence_2024}). More precisely, define
\[
\beta := \sup_{x\in X} \bar w^\top \Gamma(x),
\]
which is finite by compactness. Fix $\alpha > 0$ and set $\gamma := \beta + \alpha$, with $\alpha$ large enough that $\Gamma(X)\subset\operatorname{int} S(\gamma)$; intuitively, the cap at level $\gamma$ is placed strictly above $\Gamma(X)$, so that the slice $A$ retains the entire efficient boundary of $P$ (see~\cite[Section~6]{ararat_norm_2022}). Now, let $M\in\mathbb{R}^{q\times q}$ be symmetric positive definite, with smallest and largest eigenvalues $\lambda_{\min}(M)$ and $\lambda_{\max}(M)$ and (spectral) condition number $\kappa(M) := \lambda_{\max}(M)/\lambda_{\min}(M)$. We define the norm
\[
\|y\|_M := \sqrt{y^\top M y}, \qquad y\in\mathbb{R}^q,
\]
with dual norm $\|u\|_{M^{-1}} := \sqrt{u^\top M^{-1} u}$. Let $T := M^{1/2} \in\mathbb{R}^{q\times q}$ be the unique symmetric positive definite square root of $M$. Then $T$ is invertible, and
\[
\|y\|_M = \|Ty\|_2,
\qquad
\|u\|_{M^{-1}} = \|T^{-1} u\|_2
\]
for all $y,u\in\mathbb{R}^q$, where $\|\cdot\|_2$ denotes the Euclidean norm.

\begin{lemma}[Hausdorff distance isometry]\label{lem:hausdorff-isometry}
Let $A,B\subset\mathbb{R}^q$ be nonempty. Then
\[
\delta_H(A,B;\|\cdot\|_M) = \delta_H(TA,TB;\|\cdot\|_2),
\]
where $\delta_H(\cdot,\cdot;\|\cdot\|_M)$ denotes the Hausdorff distance with respect to $\|\cdot\|_M$.
\end{lemma}

\begin{proof}
The claim follows immediately from the isometry $\|y\|_M = \|Ty\|_2$ and the change of variables $\tilde y = Ty$, $\tilde z = Tz$ in the sup-inf definition of Hausdorff distance.
\end{proof}

\subsection{Transformed Problem}

The isometry $T=M^{1/2}$ introduced above maps the geometry of $\|\cdot\|_M$ to the Euclidean one; Lemma~\ref{lem:hausdorff-isometry} records this at the level of Hausdorff distance. In this subsection we carry the same change of variables through the data of the problem---the objective map, the ordering cone, the slice, and the norm-minimizing scalarization---so that working with $\|\cdot\|_M$ in the original coordinates is equivalent to working with the Euclidean norm in the transformed ones. This equivalence is what lets the convergence analysis of Section~\ref{sec:inner_product_invariance} reduce to the Euclidean case.

Define the transformed objective mapping $\tilde\Gamma : X\to\mathbb{R}^q$ by
\[
\tilde\Gamma(x) := T\Gamma(x),
\]
and the transformed cone $\tilde C := T C$. It is easy to see that $\tilde C$ is a closed, pointed, solid, polyhedral cone. The transformed vector optimization problem is
\begin{equation}\label{Ptilde}
\tag{$\mathrm{P}_T$}
\min \tilde\Gamma(x) \ \text{w.r.t. }\le_{\tilde C} \quad\text{s.t. } x\in X.
\end{equation}
Its upper image is
\[
\tilde P := \tilde\Gamma(X) + \tilde C = T(\Gamma(X) + C) = T P.
\]

Let $\tilde w := T^{-1}\bar w\in\mathbb{R}^q$ and define the transformed halfspace
\[
\tilde S(\tilde\gamma) := \{ y\in\mathbb{R}^q \mid \tilde w^\top y \le \tilde\gamma\},
\qquad
\tilde\gamma := \gamma.
\]
Then $T(S(\gamma)) = \tilde S(\tilde\gamma)$ and
\[
\tilde A := \tilde P \cap \tilde S(\tilde\gamma) = T(P\cap S(\gamma)) = T(A).
\]
In particular, $\tilde A$ is a nonempty convex compact set and $\tilde P = \tilde A + \tilde C$. For a fixed $v\in\mathbb{R}^q$, define $\tilde v := Tv$ and consider the primal norm-minimizing scalarization in the original coordinates based on $\|\cdot\|_M$:
\begin{equation}\label{P-v}
\tag{$P_M(v)$}
\begin{aligned}
\min_{x\in X,\ z\in\mathbb{R}^q}\ & \|z\|_M\\
\text{s.t. }& \Gamma(x) - z - v \le_C 0,\\
& \bar w^\top(v+z) \le \gamma.
\end{aligned}
\end{equation}
Its Lagrange dual can be written as
\begin{equation}\label{D-v}
\tag{$D_M(v)$}
\begin{aligned}
\max_{w\in C^+,\ \lambda\ge 0}\ & \inf_{x\in X,z\in\mathbb{R}^q}
\Bigl( \|z\|_M + w^\top(\Gamma(x)-z-v) + \lambda(\bar w^\top(v+z)-\gamma)\Bigr).
\end{aligned}
\end{equation}
\medskip
The following proposition derives the reduced form of the dual and establishes that optimal solutions of the primal--dual pair generate supporting halfspaces for $P\cap S(\gamma)$. We work first in the original coordinates (Proposition~\ref{prop:dual_support}); Propositions~\ref{prop:scalarization-transform} and~\ref{prop:algorithm-equivalence} then connect this formulation to the transformed, Euclidean one.

\begin{proposition}[Dual form and supporting halfspace]\label{prop:dual_support}
Fix $v\in\mathbb{R}^q$ and $M\succ0$. Define the Lagrangian
\[
L(x,z;w,\lambda)
:= \|z\|_{M} + w^\top(\Gamma(x)-z-v) + \lambda\big(\bar w^\top(v+z)-\gamma\big),
\]
with multipliers $w\in C^+$ and $\lambda\ge 0$, and define the dual objective
\[
\phi_M(w,\lambda)
:= \inf_{x\in X,\ z\in\mathbb{R}^q} L(x,z;w,\lambda).
\]
Then:
\begin{enumerate}
\item[(i)] (Reduction of the dual objective) the dual \eqref{D-v} reduces to the explicit form
\begin{equation}\label{eq:Dual_M}
\max_{w\in C^+,\ \lambda\ge 0}\
\Big\{\inf_{x\in X} w^\top \Gamma(x) - w^\top v + \lambda(\bar w^\top v-\gamma)\Big\}
\quad\text{s.t.}\quad \|w-\lambda\bar w\|_{M^{-1}}\le 1.
\end{equation}

\item[(ii)] (Existence of a supporting halfspace for $P\cap S(\gamma)$)
Assume $(x^\star,z^\star)$ is an optimal solution of \eqref{P-v} and $(w^\star,\lambda^\star)$ is an optimal solution of \eqref{eq:Dual_M} such that strong duality holds, and set $y^\star := v+z^\star$. Then $y^\star\in P\cap S(\gamma)$, and the halfspace $H^\star := \{y\in\mathbb{R}^q:\ w^{\star\top} y \ge w^{\star\top} y^\star\}$ contains $P\cap S(\gamma)$ and supports it at $y^\star$ (when $w^\star\neq 0$); that is, $w^{\star\top} y \ge w^{\star\top} y^\star$ for all $y\in P\cap S(\gamma)$.
\end{enumerate}
\end{proposition}

\begin{proof}
(i) Write the Lagrangian as
\[
L(x,z;w,\lambda)
=
\underbrace{w^\top \Gamma(x)}_{\text{$x$-part}}
\;+\;
\underbrace{\big(\|z\|_{M} + (\lambda\bar w-w)^\top z\big)}_{\text{$z$-part}}
\;+\;
\underbrace{\big(-w^\top v + \lambda(\bar w^\top v-\gamma)\big)}_{\text{constants}}.
\]
Taking the infimum over $x\in X$ yields $\inf_{x\in X} w^\top \Gamma(x)$.
For the $z$-part, recall the Fenchel conjugate of the norm: the conjugate of $z\mapsto \|z\|_M$ is the indicator function of the dual unit ball $\{a:\ \|a\|_{M^{-1}}\le 1\}$. Equivalently,
\[
\inf_{z\in\mathbb{R}^q}\big(\|z\|_{M} + a^\top z\big)
=
\begin{cases}
0, & \text{if }\ \|a\|_{M^{-1}}\le 1,\\
-\infty, & \text{otherwise.}
\end{cases}
\]
Applying this with $a=\lambda\bar w-w$ gives, when $\|w-\lambda\bar w\|_{M^{-1}}\le 1$, $\phi_M(w,\lambda)=\inf_{x\in X} w^\top \Gamma(x) - w^\top v + \lambda(\bar w^\top v-\gamma)$ (and $\phi_M(w,\lambda)=-\infty$ otherwise), whence \eqref{eq:Dual_M}, where $\|a\|_{M^{-1}}:=\sqrt{a^\top M^{-1}a}$.

(ii) Feasibility of $(x^\star,z^\star)$ for \eqref{P-v} implies $\Gamma(x^\star)-z^\star-v\in -C$, i.e., $v+z^\star=\Gamma(x^\star)+c^\star$ for some $c^\star\in C$, hence $y^\star=v+z^\star\in P$.
The constraint $\bar w^\top(v+z^\star)\le \gamma$ gives $y^\star\in S(\gamma)$, thus $y^\star\in P\cap S(\gamma)$.

Let $y\in P\cap S(\gamma)$ be arbitrary; write $y=\Gamma(x)+c$ with $x\in X$ and $c\in C$.
Since $w^\star\in C^+$ and $c\in C$, we have $w^{\star\top}c\ge 0$, hence
\[
w^{\star\top}y \;\ge\; w^{\star\top}\Gamma(x)\;\ge\;\inf_{x'\in X} w^{\star\top}\Gamma(x').
\]

\begin{revblock}
By the saddle-point optimality of the primal--dual pair---the assumed optimality of $(x^\star,z^\star)$ and $(w^\star,\lambda^\star)$ together with strong duality makes $(x^\star,z^\star;w^\star,\lambda^\star)$ a saddle point of the Lagrangian---$x^\star$ minimizes the $x$-part $w^{\star\top}\Gamma(\cdot)$ of the Lagrangian over $X$, whence $\inf_{x'\in X} w^{\star\top}\Gamma(x')=w^{\star\top}\Gamma(x^\star)$. Complementary slackness for the cone constraint gives $(w^\star)^\top c^\star=0$, where $c^\star\in C$ is the slack in $y^\star=v+z^\star=\Gamma(x^\star)+c^\star$; hence
\[
w^{\star\top}y^\star=w^{\star\top}\Gamma(x^\star)=\inf_{x'\in X} w^{\star\top}\Gamma(x').
\]
Combining this with the display above yields $w^{\star\top}y\ge w^{\star\top}y^\star$ for all $y\in P\cap S(\gamma)$, which proves~(ii).
\end{revblock}
\end{proof}

We now return to the transformed problem. The following proposition shows that, under the isometry $T$, the original-coordinate scalarization \eqref{P-v} is exactly a Euclidean norm-minimizing scalarization on the transformed data, with its dual transforming correspondingly.

\begin{proposition}[Scalarization transform]\label{prop:scalarization-transform}
Let $\tilde v := Tv$ and $\tilde z := Tz$. Then \eqref{P-v} is equivalent to the Euclidean norm--based scalarization
\begin{equation}\label{P2-v}
\tag{$\tilde P_2(\tilde v)$}
\begin{aligned}
\min_{x\in X,\ \tilde z\in\mathbb{R}^q}\ & \|\tilde z\|_2\\
\text{s.t. }& \tilde\Gamma(x) - \tilde z - \tilde v \le_{\tilde C} 0,\\
& \tilde w^\top(\tilde v+\tilde z)\le \tilde\gamma.
\end{aligned}
\end{equation}
Moreover, the dual problem \eqref{D-v} is equivalent to the Euclidean dual of \eqref{P2-v}.
\end{proposition}

\begin{proof}
Since $T=M^{1/2}$ is invertible, the map $(x,z)\mapsto(x,Tz)$ is a bijection on $X\times\mathbb{R}^q$. We verify it maps the constraints and objective of \eqref{P-v} to those of \eqref{P2-v}. For the objective, $M=T^2$ gives
$\|z\|_M = \sqrt{z^\top T^2 z} = \|Tz\|_2 = \|\tilde z\|_2$. For the cone constraint, since $\tilde C=TC$ and $T$ is invertible,
\[
\Gamma(x)-z-v\in -C
\;\iff\;
T\bigl(\Gamma(x)-z-v\bigr)\in -\tilde C
\;\iff\;
\tilde\Gamma(x)-\tilde z-\tilde v \in -\tilde C.
\]
For the slice constraint, since $T$ is symmetric,
$\bar w^\top(v+z) = (T^{-1}\bar w)^\top T(v+z) = \tilde w^\top(\tilde v+\tilde z)$,
so the constraint is preserved since $\tilde\gamma=\gamma$. Hence $(x,z)$ is feasible for \eqref{P-v} with objective $\|z\|_M$ if and only if $(x,\tilde z)$ is feasible for \eqref{P2-v} with objective $\|\tilde z\|_2=\|z\|_M$. Hence, the primal problems are equivalent. On the other hand, for the dual problem, we write $\hat w:=T^{-1}w$ for the transformed dual variable (to distinguish it from the transformed slice direction $\tilde w=T^{-1}\bar w$).
The Lagrangian of \eqref{P-v} is
\[
L(x,z;\,w,\lambda) \;=\; \|z\|_M + w^\top\bigl(\Gamma(x)-z-v\bigr) + \lambda\bigl(\bar w^\top(v+z)-\gamma\bigr),
\]
with $w\in C^+$ and $\lambda\ge 0$, and the dual \eqref{D-v} is
$\sup_{w\in C^+,\,\lambda\ge 0}\;\inf_{x\in X,\,z\in\mathbb{R}^q} L(x,z;w,\lambda)$. We substitute $z=T^{-1}\tilde z$ in the infimum and $w=T\hat w$ in the supremum, keeping $\lambda$ unchanged. The dual cone of $\tilde C=TC$ is $\tilde C^+=T^{-1}C^+$; indeed, since $T$ is symmetric,
\[
u\in\tilde C^+
\;\iff\;
u^\top(Tc)\ge 0\;\forall\,c\in C
\;\iff\;
(Tu)^\top c\ge 0\;\forall\,c\in C
\;\iff\;
Tu\in C^+.
\]
Hence $w=T\hat w\in C^+$ if and only if $\hat w\in\tilde C^+$. Now we compute each term of $L(x,\,T^{-1}\tilde z;\,T\hat w,\,\lambda)$. The first term gives $\|T^{-1}\tilde z\|_M = \|\tilde z\|_2$ as in the primal part.
For the second term, since $T$ is symmetric,
\[
(T\hat w)^\top\bigl(\Gamma(x)-T^{-1}\tilde z-v\bigr)
= \hat w^\top T\bigl(\Gamma(x)-T^{-1}\tilde z-v\bigr)
= \hat w^\top\bigl(\tilde\Gamma(x)-\tilde z-\tilde v\bigr).
\]
The third term gives
$\lambda\bigl(\bar w^\top(v+T^{-1}\tilde z)-\gamma\bigr) = \lambda\bigl(\tilde w^\top(\tilde v+\tilde z)-\tilde\gamma\bigr)$
by the slice constraint calculation from the primal part. Combining,
\[
L\bigl(x,\,T^{-1}\tilde z;\,T\hat w,\,\lambda\bigr)
= \|\tilde z\|_2 + \hat w^\top\bigl(\tilde\Gamma(x)-\tilde z-\tilde v\bigr) + \lambda\bigl(\tilde w^\top(\tilde v+\tilde z)-\tilde\gamma\bigr),
\]
which is precisely the Lagrangian of \eqref{P2-v} with Euclidean norm and dual variables $(\hat w,\lambda)\in\tilde C^+\times[0,\infty)$. Since the bijections $z\leftrightarrow\tilde z=Tz$ on primal variables and $w\leftrightarrow\hat w=T^{-1}w$ on dual variables preserve both the Lagrangian and the feasible sets, the dual problems have equal optimal values and corresponding optimal solutions.
\end{proof}

\begin{proposition}[Algorithm equivalence]\label{prop:algorithm-equivalence}
Consider the outer approximation algorithm of \cite[Algorithm 1]{ararat_convergence_2024}, where in each iteration a problem of the form \eqref{P-v} is solved and supporting halfspaces of $A$ are generated via dual optimal solutions.  Let $(A_k)_{k\ge 0}$ be the sequence of outer approximations in the original coordinates, and let
\[
\tilde A_k := T A_k,\quad k\ge 0.
\]
Then $(\tilde A_k)_{k\ge 0}$ is precisely the sequence of outer approximations obtained by applying the same algorithm to the transformed problem \eqref{Ptilde}, using the Euclidean scalarization \eqref{P2-v}.
\end{proposition}

\begin{proof}
We note that $T$ maps supporting halfspaces of $A$ to those of $\tilde A=TA$. Indeed, using $h_A(w)=\inf_{y\in A} w^\top y$ and the identity $h_{\tilde A}(T^{-1}w) = \inf_{\tilde y\in \tilde A}(T^{-1}w)^\top\tilde y = \inf_{y\in A} w^\top y = h_A(w)$, we get
\begin{equation}\label{eq:halfspace_transform}
T\bigl(\mathcal{H}(w,A)\bigr)
= \bigl\{\tilde y: (T^{-1}w)^\top \tilde y \ge h_A(w)\bigr\}
= \mathcal{H}(T^{-1}w,\,\tilde A).
\end{equation}

We now proceed by induction on $k$. The initial outer approximation is $A_0 = \bigl(\bigcap_{j=1}^I \mathcal{H}(\omega_j,A)\bigr)\cap S(\gamma)$ for some directions $\omega_1,\ldots,\omega_I\in C^+\setminus\{0\}$ that are obtained from weighted-sum scalarizations $\min_{x\in X}\omega_j^\top\Gamma(x)$. Since $T$ commutes with finite intersections,
\[
TA_0
= \bigcap_{j=1}^I T\bigl(\mathcal{H}(\omega_j,A)\bigr) \;\cap\; T\bigl(S(\gamma)\bigr)
= \bigcap_{j=1}^I \mathcal{H}(T^{-1}\omega_j,\,\tilde A) \;\cap\; \tilde S(\tilde\gamma),
\]
using \eqref{eq:halfspace_transform} and $T(S(\gamma))=\tilde S(\tilde\gamma)$.
This is precisely the initial outer approximation $\tilde A_0$ for the transformed problem \eqref{Ptilde}: the directions $T^{-1}\omega_j\in\tilde C^+\setminus\{0\}$ are the transformed generators, and the weighted-sum scalarizations yield the same optimizers $x_j\in X$ since $(T^{-1}\omega_j)^\top\tilde\Gamma(x)=\omega_j^\top\Gamma(x)$. Hence $\tilde A_0=TA_0$. Now, assume $\tilde A_k=TA_k$. At iteration $k$, the algorithm performs three operations. The original algorithm selects $v_k\in\arg\max_{v\in\operatorname{ext}(A_k)}\|z^v\|_M$, where $z^v$ is the optimal displacement from \eqref{P-v}. Since the invertible linear map $T$ preserves extreme points (for any convex set $Q$: if $Tv=\lambda Ty_1+(1-\lambda)Ty_2$ with $Ty_1,Ty_2\in T Q$, then $v=\lambda y_1+(1-\lambda)y_2$ with $y_1,y_2\in Q$), we have $\operatorname{ext}(\tilde A_k)=T(\operatorname{ext}(A_k))$. By Proposition~\ref{prop:scalarization-transform}, $\|z^v\|_M=\|\tilde z^{Tv}\|_2$ for each $v\in\operatorname{ext}(A_k)$, so $\tilde v_k:=Tv_k\in\arg\max_{\tilde v\in\operatorname{ext}(\tilde A_k)}\|\tilde z^{\tilde v}\|_2$, which is the selection rule for the transformed algorithm. By Proposition~\ref{prop:scalarization-transform}, solving \eqref{P-v} at $v_k$ with norm $\|\cdot\|_M$ yields a primal--dual pair $(x_k,z_k;\,w_k,\lambda_k)$ corresponding to the pair $(x_k,Tz_k;\,T^{-1}w_k,\lambda_k)$ for the Euclidean norm-based scalarization \eqref{P2-v} at $\tilde v_k$.
The boundary point $y_k=v_k+z_k$ maps to $\tilde y_k=Ty_k=\tilde v_k+\tilde z_k$.
The cut halfspace $H_{k+1}=\{y:a_k^\top y\ge a_k^\top y_k\}$, where $a_k$ is a linear function of the dual solution $(w_k,\lambda_k,\bar w)$, transforms as
\[
T(H_{k+1}) = \{\tilde y:(T^{-1}a_k)^\top\tilde y \ge (T^{-1}a_k)^\top\tilde y_k\}.
\]
Since $T^{-1}$ is linear, $T^{-1}a_k$ is the same linear function evaluated at the transformed dual solution $(T^{-1}w_k,\lambda_k,\tilde w)$, so $T(H_{k+1})=\tilde H_{k+1}$. Using the inductive hypothesis and the intersection identity,
\[
TA_{k+1} = T(A_k\cap H_{k+1}) = TA_k\cap TH_{k+1} = \tilde A_k\cap\tilde H_{k+1} = \tilde A_{k+1}.\qedhere
\]
\end{proof}


\section{Inner-Product Norm Invariance and Improved Convergence Rate}\label{sec:inner_product_invariance}

The convergence theory of \cite{ararat_convergence_2024} is built on the Euclidean norm, whose improved rate $O(k^{2/(1-q)})$ exploits its special geometry. The Euclidean norm, however, may not reflect the geometry of a given problem---for instance when the objectives have very different scales---so one would like to measure distances in a problem-adapted way. The inner-product norms $\|\cdot\|_M$ ($M\succ 0$) provide this flexibility, and they are the metrics that the adaptive-metric framework of Section~\ref{sec:adaptive_framework} later tunes. This section shows that the flexibility is free. Every inner-product norm achieves the same rate as the Euclidean one.

Recall from Section~\ref{sec:preliminaries} the compact slice $A := P\cap S(\gamma)$ and its transform $\tilde A := T(A)$. For a convex compact set $B\subset\mathbb{R}^q$, denote by $R(B)$ the radius of the smallest Euclidean ball containing $B$. Let $\pi_q$ be the volume of the Euclidean unit ball in $\mathbb{R}^q$, and define
\[
\bar\lambda(B) := 16\,R(B)\left(\frac{q\,\pi_q}{\pi_{q-1}}\right)^{\frac{2}{q-1}}.
\]

The following is a direct consequence of the Euclidean convergence result in \cite[Thm.~7.2]{ararat_convergence_2024} and Propositions~\ref{prop:scalarization-transform}--\ref{prop:algorithm-equivalence}.

\begin{theorem}[Improved convergence rate under inner-product norms]\label{thm:inner-product-rate}
Assume Assumption~\ref{ass:standing}, and let $\|\cdot\|_M$ be an inner-product norm induced by $M\succ 0$. Let $(A_k)_{k\ge 0}$ be the sequence of outer approximations of the compact slice $A=P\cap S(\gamma)$ generated by the norm-minimization--based algorithm of \cite{ararat_convergence_2024}, where the scalarizations and Hausdorff distance are based on $\|\cdot\|_M$.

Then the following statements hold.
\begin{enumerate}
\item For every $\varepsilon>0$ there exists $K\in\mathbb{N}$ such that for all $k\ge K$,
\[
\delta_H(A_k,A;\|\cdot\|_M)
\;\le\; (1+\varepsilon)\,\bar\lambda(\tilde A)\,k^{\frac{2}{1-q}},
\]
where $\tilde A = T(A)=TP\cap T(S(\gamma))$.
\item In particular,
\[
\delta_H(A_k,A;\|\cdot\|_M) = O\bigl(k^{\frac{2}{1-q}}\bigr)
\quad\text{as }k\to\infty.
\]
\end{enumerate}
\end{theorem}

\begin{proof}
It suffices to prove the bound for $\varepsilon\in(0,1)$; for $\varepsilon\ge1$ it then follows from the case $\varepsilon'=1/2$, since $1+\varepsilon'\le1+\varepsilon$. By Proposition~\ref{prop:algorithm-equivalence}, $(\tilde A_k)_{k\ge 0}$ with $\tilde A_k := TA_k$ is the sequence of outer approximations of the transformed slice $\tilde A=TA$ generated by the Euclidean version of the algorithm applied to the transformed problem \eqref{Ptilde} using the Euclidean scalarization \eqref{P2-v}. By \cite[Thm.~7.2]{ararat_convergence_2024} (which applies for $\varepsilon\in(0,1)$), there exists $K\in\mathbb{N}$ such that
\[
\delta_H(\tilde A_k,\tilde A;\|\cdot\|_2)
\;\le\; (1+\varepsilon)\,\bar\lambda(\tilde A)\,k^{\frac{2}{1-q}}
\quad\text{for all }k\ge K
\]
(\cite[Thm.~7.2]{ararat_convergence_2024} bounds a $(k-1)$-indexed iterate; applied at index $k+1$ it gives the bound $(1+\varepsilon)\bar\lambda(\tilde A)(k+1)^{2/(1-q)}$, and since $2/(1-q)<0$ we have $(k+1)^{2/(1-q)}\le k^{2/(1-q)}$, yielding the display above). Using Lemma~\ref{lem:hausdorff-isometry} with the sets $A_k$ and $A$, we obtain
\[
\delta_H(A_k,A;\|\cdot\|_M)
= \delta_H(TA_k,TA;\|\cdot\|_2)
= \delta_H(\tilde A_k,\tilde A;\|\cdot\|_2),
\]
which yields the claim.
\end{proof}

\begin{remark}[Role of the metric in convergence bounds]
Theorem~\ref{thm:inner-product-rate} shows that the improved exponent $2/(1-q)$ is not tied to the standard Euclidean structure but to the Hilbertian geometry of the norm. The choice of $M$ enters only through the multiplicative constant: the theorem gives $\delta_H(A_k,A;\|\cdot\|_M)\le(1+\varepsilon)\,\bar\lambda(\tilde A)\,k^{2/(1-q)}$, where the constant $\bar\lambda(\tilde A)$ depends on the transformed slice $\tilde A=TA$ (through its circumradius). A badly conditioned $M$---one with large condition number $\kappa(M)=\lambda_{\max}(M)/\lambda_{\min}(M)$---distorts $\tilde A$ and inflates $\bar\lambda(\tilde A)$, enlarging the constant, even though the asymptotic exponent $2/(1-q)$ is unchanged.
\end{remark}


\section{Adaptive-Metric Framework}\label{sec:adaptive_framework}

While Section~\ref{sec:inner_product_invariance} establishes that any fixed inner-product norm achieves the optimal convergence exponent, a natural question arises: \rev{can iteration-dependent metrics adapt to the problem geometry while preserving the convergence guarantee? We show in Section~\ref{sec:convergence} that they retain the improved exponent $2/(1-q)$ (Theorem~\ref{thm:rate}), the conditioning of the metric governing only the multiplicative constant.} In this section, we introduce an adaptive-metric framework in which the scalarization norm varies across iterations while the approximation error is measured in a fixed Euclidean norm.

\subsection{Framework Design: Fixed Euclidean Accuracy}

Throughout this section, the approximation quality is measured in the fixed Euclidean norm $\|\cdot\|_2$ on $\mathbb{R}^q$. The algorithm employs an iteration-dependent inner product norm $\|\cdot\|_{M_k}$ only inside the scalarization subproblems used to generate supporting halfspaces.

Let $B\subseteq \mathbb{R}^q$ be a closed convex set, and let $B_k\supseteq B$ be a polyhedral outer approximation that has at least one vertex and the same recession cone, $\mathrm{rec}(B_k)=\mathrm{rec}(B)$; this equal-recession-cone condition guarantees that $\delta_H(B_k,B;\|\cdot\|_2)$ is finite and attained at a vertex of $B_k$. In the convex vector optimization application, $P=\Gamma(X)+C$ is the (unbounded) upper image, with $\mathrm{rec}(P)=C$; equivalently, and as used in Theorems~\ref{thm:euclid_convergence_adaptive} and~\ref{thm:rate}, one may work with the compact slice $A=P\cap S(\gamma)$ and \rev{compact polyhedral outer approximations $A_k\supseteq A$ generated directly by the algorithm below} (for which $\mathrm{rec}=\{0\}$). \rev{The adaptive algorithm of this section operates on this compact-slice formulation: each generated halfspace contains $A$ (Lemma~\ref{lem:cut_removes_vertex})---but need not contain the unbounded image $P$, since the combined cut-normal $g_i$ need not lie in $C^+$---so $A\subseteq A_{k+1}\subseteq A_k$, and the error is measured as $\max_{v\in \operatorname{ext}(A_k)}d_2(v,A)$.}

\begin{revblock}
We define the Euclidean outer Hausdorff error of the compact outer approximation $A_k$ relative to the slice $A$ by
\[
\mathrm{err}_2(A_k,A)
:= \delta_H(A_k,A;\|\cdot\|_2)
= \max_{v\in \operatorname{ext}(A_k)} d_2(v,A),
\qquad
d_2(v,A):=\inf_{y\in A}\|v-y\|_2,
\]
where $\operatorname{ext}(A_k)$ denotes the vertex set of $A_k$.
\end{revblock} 
\begin{revblock}
Assume that iteration $k$ ($k\ge 0$) produces a valid supporting halfspace
\[
H_{k+1} := \{y\in \mathbb{R}^q:\; g_{k+1}^\top y \ge g_{k+1}^\top y_k\},
\qquad
g_{k+1} := w_{k+1}-\lambda_{k+1}\bar w,
\]
that contains the compact slice $A=P\cap S(\gamma)$ (i.e., $A\subseteq H_{k+1}$) and supports it at the boundary point $y_k$ generated at iteration $k$. Here $g_{k+1}$ is the \emph{combined cut-normal} formed from the scalarization dual pair $(w_{k+1},\lambda_{k+1})\in C^+\times[0,\infty)$, exactly as in Lemma~\ref{lem:cut_removes_vertex}; in general $g_{k+1}\notin C^+$. Moreover $g_{k+1}\neq0$: by Lemma~\ref{lem:cut_removes_vertex} it equals $M z/\|z\|_{M}$, where $M$ and $z$ are the metric and optimal displacement of the scalarization that generates the cut, and this is nonzero because a cut is generated only when the reference vertex lies outside $A$, i.e.\ $z\neq0$. Hence the normalized cut-normals below are well defined. Define, for each $i\ge 1$, the normalized cut-normals
\[
u_i := \frac{g_i}{\|g_i\|_2}.
\]
\end{revblock}
Fix a regularization parameter $\varepsilon_0>0$ and define, for $k\ge 1$, the symmetric positive definite $q\times q$ matrix
\begin{equation*}
M_k := \varepsilon_0 I \;+\; \frac{1}{k}\sum_{i=1}^k u_i u_i^\top,
\qquad
\|d\|_{M_k}:=\sqrt{d^\top M_k d}.
\end{equation*}
\rev{We set $M_0:=\varepsilon_0 I$. At each iteration, the cut-normal $u_{k+1}$ is generated by solving the scalarization with the \emph{current} metric $M_k$, and the metric is then updated to $M_{k+1}$ by the running average above. Thus each $M_k$ depends only on the cut-normals $u_1,\dots,u_k$ produced at earlier iterations, so the recurrence $M_k\mapsto u_{k+1}\mapsto M_{k+1}$ is well defined and non-circular.}

\begin{lemma}[Uniform spectral bounds and uniform norm equivalence]\label{lem:spec_bounds}
For all $k\ge 0$,
\[
\varepsilon_0 I \preceq M_k \preceq (\varepsilon_0+1)I.
\]
Consequently, for all $d\in\mathbb{R}^q$ and all $k\ge 0$,
\[
\sqrt{\varepsilon_0}\,\|d\|_2 \le \|d\|_{M_k} \le \sqrt{\varepsilon_0+1}\,\|d\|_2,
\]
and
\[
\frac{1}{\sqrt{\varepsilon_0+1}}\,\|d\|_2 \le \|d\|_{M_k^{-1}} \le \frac{1}{\sqrt{\varepsilon_0}}\,\|d\|_2.
\]
\end{lemma}

\begin{proof}
\rev{For $k=0$, $M_0=\varepsilon_0 I$ satisfies the bounds directly; assume $k\ge 1$.} Since $\|u_i\|_2=1$, we have $0\preceq u_i u_i^\top \preceq I$ for all $i$.
Averaging preserves the Loewner order, hence
\[
0\preceq \frac{1}{k}\sum_{i=1}^k u_i u_i^\top \preceq I.
\]
Adding $\varepsilon_0 I$ yields $\varepsilon_0 I\preceq M_k\preceq (\varepsilon_0+1)I$.
The norm inequalities follow from the Rayleigh quotient bounds:
$\lambda_{\min}(M_k)\|d\|_2^2 \le d^\top M_k d \le \lambda_{\max}(M_k)\|d\|_2^2$
and the corresponding inequalities for $M_k^{-1}$.
\end{proof}

\begin{corollary}[Bridge from adaptive metric to Euclidean distance]\label{cor:bridge}
For any nonempty set $S\subseteq\mathbb{R}^q$, any $v\in\mathbb{R}^q$, and any $k\ge 1$,
\[
d_2(v,S)\ \le\ \frac{1}{\sqrt{\varepsilon_0}}\, d_{M_k}(v,S),
\qquad
d_{M_k}(v,S):=\inf_{y\in S}\|v-y\|_{M_k}.
\]
\begin{revblock}In particular, if $d_{M_k}(v,A)\le \sqrt{\varepsilon_0}\,\varepsilon$, then $d_2(v,A)\le \varepsilon$.\end{revblock}
\end{corollary}

\begin{proof}
By Lemma~\ref{lem:spec_bounds}, $\|v-y\|_2 \le \frac{1}{\sqrt{\varepsilon_0}}\|v-y\|_{M_k}$ for all $y$.
Taking the infimum over $y\in S$ establishes the claim.
\end{proof}

\subsection{Adaptive-Metric Outer Approximation Algorithm}

We consider the convex vector optimization setting where the upper image is $P=\Gamma(X)+C$ (under Assumption~\ref{ass:standing}), and we generate compact polyhedral outer approximations $A_k\supseteq A$ of the compact slice $A=P\cap S(\gamma)$ by iteratively adding supporting halfspaces of $A$ obtained from norm-minimization scalarizations. The base algorithm follows the norm-minimization outer approximation framework of \cite{ararat_convergence_2024}; our modification is to use an adaptive metric $\|\cdot\|_{M_k}$ in the scalarization at iteration $k$, while measuring the approximation error in the fixed Euclidean norm. The algorithm proceeds as follows. Starting from an initial outer approximation \rev{$A_0 \supseteq A$ of the compact slice $A$}, at each iteration $k$ we: (i) select a vertex $v_k$ from the current approximation \rev{$A_k$}; (ii) construct the adaptive metric $M_k = \varepsilon_0 I + \frac{1}{k}\sum_{i=1}^k u_i u_i^\top$ \rev{(for $k\ge1$; $M_0=\varepsilon_0 I$)} from the normalized \rev{combined cut-normals $u_i = g_i/\|g_i\|_2$ with $g_i = w_i-\lambda_i\bar w$}; (iii) solve the $M_k$-norm minimization scalarization at $v_k$ to obtain a boundary point \rev{$y_k \in A$} and \rev{combined cut-normal $g_{k+1} = w_{k+1}-\lambda_{k+1}\bar w$}; and (iv) update \rev{$A_{k+1} = A_k \cap H_{k+1}$} where \rev{$H_{k+1} = \{y : g_{k+1}^\top y \geq g_{k+1}^\top y_k\}$}. The algorithm terminates when \rev{$\max_{v \in \operatorname{ext}(A_k)} d_2(v, A) \leq \varepsilon$}. For full algorithmic details, we refer to Algorithm~1 of \cite{ararat_convergence_2024}; our modification is the adaptive metric $M_k$ construction.

\begin{remark}[Practical stopping criterion]\label{rem:practical_stop}
In the full-enumeration runs the algorithm stops on the computable \emph{certified Euclidean upper proxy} $\widehat E_k=\max_{v\in \operatorname{ext}(A_k)}\|z^v\|_2\le\varepsilon$ (the $q=4$ probe/hybrid strategies instead use the candidate-set residual of Section~\ref{sec:vertex-strategies}), where $z^v$ is the optimal displacement of the $M_k$-scalarization at $v$; since $\widehat E_k\ge \max_{v\in \operatorname{ext}(A_k)} d_2(v,A)$ (Corollary~\ref{cor:Ehat} below), this guarantees the Euclidean criterion $\max_{v\in \operatorname{ext}(A_k)} d_2(v,A)\le\varepsilon$. A looser alternative is the metric-norm surrogate $\max_{v\in \operatorname{ext}(A_k)} d_{M_k}(v,A)\le \sqrt{\varepsilon_0}\,\varepsilon$, which also implies $\max_{v\in \operatorname{ext}(A_k)} d_2(v,A)\le \varepsilon$ by Corollary~\ref{cor:bridge}.
\end{remark}

\begin{remark}[Adaptive substitution]\label{rem:adaptive_substitution}
In the adaptive-metric algorithm, $M$ is replaced by $M_k$ at iteration $k$. The only change in Proposition~\ref{prop:dual_support} is the dual constraint $\|w-\lambda\bar w\|_{M^{-1}}\le 1$, which becomes $\|w-\lambda\bar w\|_{M_k^{-1}}\le 1$. All other parts of the supporting-halfspace argument remain unchanged.
\end{remark}

\subsection{Vertex Selection and Distance Proxy}

At iteration $k$, let $A_k$ denote the current polyhedral outer approximation of the slice $A$. For each vertex $v\in \operatorname{ext}(A_k)$ we solve the scalarization $(\mathrm P_{M_k}(v))$ and obtain an optimal $z^v\in\mathbb{R}^q$, and we then select
\begin{equation}\label{eq:vertex_rule}
v^k \in \arg\max\Big\{\ \|z^v\|_{M_k}\ :\ v\in \operatorname{ext}(A_k)\ \Big\}.
\end{equation}

\begin{lemma}[Selection proxy for Euclidean distance]\label{lem:zk_proxy}
Let $A := P\cap S(\gamma)$ where $P:=\Gamma(X)+C$ and $S(\gamma)$ is the fixed slicing halfspace. Fix $k\ge 0$ and $M_k\succ0$. For any $v\in\operatorname{ext}(A_k)$, let $(x^v,z^v)$ be an optimal solution of $(\mathrm P_{M_k}(v))$ and set $y^v:=v+z^v\in A$. Then
\[
d_{M_k}(v,A)=\|z^v\|_{M_k}.
\]
Moreover, if $M_k$ satisfies $\varepsilon_0 I\preceq M_k \preceq (\varepsilon_0+1)I$, then
\[
d_2(v,A)\ \le\ \frac{1}{\sqrt{\varepsilon_0}}\,\|z^v\|_{M_k}.
\]
\end{lemma}

\begin{proof}
By feasibility of $(x^v,z^v)$, we have $y^v=v+z^v\in A$, hence $d_{M_k}(v,A)\le \|v-y^v\|_{M_k}=\|z^v\|_{M_k}$.
Conversely, if $y\in A$, then $y=\Gamma(x)+c$ for some $x\in X$ and $c\in C$.
Set $z:=y-v$. Then $(x,z)$ is feasible for $(\mathrm P_{M_k}(v))$ and $\|z\|_{M_k}=\|y-v\|_{M_k}$.
Optimality of $z^v$ yields $\|z^v\|_{M_k}\le \|y-v\|_{M_k}$ for all $y\in A$, hence $\|z^v\|_{M_k}\le d_{M_k}(v,A)$. Therefore equality holds: $d_{M_k}(v,A)=\|z^v\|_{M_k}$. By Lemma~\ref{lem:spec_bounds}, the Euclidean bound follows from the uniform norm equivalence $\|d\|_2 \le \frac{1}{\sqrt{\varepsilon_0}}\|d\|_{M_k}$:
\[
d_2(v,A)=\inf_{y\in A}\|v-y\|_2
\le \frac{1}{\sqrt{\varepsilon_0}}\inf_{y\in A}\|v-y\|_{M_k}
= \frac{1}{\sqrt{\varepsilon_0}}\,\|z^v\|_{M_k}.
\]
\end{proof}

\begin{corollary}[Certified Euclidean upper proxy]\label{cor:Ehat}
Fix $k\ge 0$, assume $\varepsilon_0 I\preceq M_k\preceq(\varepsilon_0+1)I$, and for each $v\in\operatorname{ext}(A_k)$ let $z^v$ be an optimal displacement of $(\mathrm P_{M_k}(v))$. Set
\[
E_k := \max_{v\in\operatorname{ext}(A_k)} d_2(v,A),
\qquad
\widehat E_k := \max_{v\in\operatorname{ext}(A_k)} \|z^v\|_2 .
\]
Then
\[
E_k \;\le\; \widehat E_k \;\le\; \sqrt{\frac{\lambda_{\max}(M_k)}{\lambda_{\min}(M_k)}}\;E_k \;\le\; \sqrt{\frac{\varepsilon_0+1}{\varepsilon_0}}\;E_k .
\]
Thus $\widehat E_k$---the largest Euclidean length of the $M_k$-optimal displacements---is a computable upper bound on the true Euclidean error $E_k$, tight up to the metric-quality factor $\theta_k^{-1}=\sqrt{\lambda_{\max}(M_k)/\lambda_{\min}(M_k)}$ (Definition~\ref{def:theta_k}). It is the quantity reported for the adaptive metric in the full-enumeration runs of Section~\ref{sec:experiments}.
\end{corollary}

\begin{proof}
Fix a vertex $v$. Feasibility of $z^v$ gives $v+z^v\in A$, so $d_2(v,A)\le\|z^v\|_2$; taking the maximum over $v$ yields $E_k\le\widehat E_k$. For the reverse bound, let $z_2^v$ be a Euclidean nearest-point displacement, so $v+z_2^v\in A$ and $\|z_2^v\|_2=d_2(v,A)$. Since $z^v$ minimizes $\|\cdot\|_{M_k}$ over displacements landing in $A$, Lemma~\ref{lem:spec_bounds} gives
\[
\|z^v\|_2
\le \frac{1}{\sqrt{\lambda_{\min}(M_k)}}\,\|z^v\|_{M_k}
\le \frac{1}{\sqrt{\lambda_{\min}(M_k)}}\,\|z_2^v\|_{M_k}
\le \sqrt{\frac{\lambda_{\max}(M_k)}{\lambda_{\min}(M_k)}}\;d_2(v,A).
\]
Taking the maximum over $v$ yields $\widehat E_k\le\sqrt{\lambda_{\max}(M_k)/\lambda_{\min}(M_k)}\,E_k$; finally $\lambda_{\min}(M_k)\ge\varepsilon_0$ and $\lambda_{\max}(M_k)\le\varepsilon_0+1$ give the last inequality.
\end{proof}

\begin{remark}\label{rem:vertex_rule_consistency}
Vertex selection, given by~\eqref{eq:vertex_rule}, coincides with the baseline rule of \cite{ararat_norm_2022,ararat_convergence_2024} when the norm in the scalarization objective is fixed. In the adaptive-metric setting, \eqref{eq:vertex_rule} selects a vertex maximizing the $M_k$-distance to $A$. By Lemmas~\ref{lem:spec_bounds} and~\ref{lem:zk_proxy}, this also controls the Euclidean distance to $A$ up to the constant $1/\sqrt{\varepsilon_0}$, which is what is needed for the accuracy guarantees of Theorem~\ref{thm:baseline_correctness_slice}, stated in the fixed Euclidean norm.
\end{remark}


\section{Convergence Analysis}\label{sec:convergence}

This section establishes the convergence results for the adaptive-metric algorithm. We first prove baseline correctness (Theorem~\ref{thm:baseline_correctness_slice}) and the key lemma showing that each cut removes the selected vertex (Lemma~\ref{lem:cut_removes_vertex}), and then establish qualitative Euclidean convergence (Theorem~\ref{thm:euclid_convergence_adaptive}), independently of any dispersion hypothesis. \rev{The remainder of the section derives the improved $O(k^{2/(1-q)})$ convergence rate of the adaptive-metric algorithm (Theorem~\ref{thm:rate}), by adapting the Euclidean packing argument to the iteration-dependent metric.}


\begin{theorem}[Baseline correctness (fixed Euclidean accuracy on the slice)]
\label{thm:baseline_correctness_slice}
Assume Assumption~\ref{ass:standing} and the slice construction of Section~\ref{sec:preliminaries}. Let $A := P\cap S(\gamma)$ and let $\{A_k\}_{k\ge 0}$ be generated by iteratively intersecting the outer approximation with the supporting halfspaces obtained from the scalarizations $(\mathrm P_{M_k}(v))$, with $M_k\succ 0$, and let each generated halfspace $H_{k+1}$ satisfy $A\subseteq H_{k+1}$. Then:
\begin{enumerate}
\item[(i)] (Outer-approximation invariance) For all $k\ge 0$,
\[
A\ \subseteq\ A_{k+1}\ \subseteq\ A_{k}.
\]

\item[(ii)] (Correctness upon Euclidean termination)
If for some $\bar k$ one has $\mathrm{err}_2(A_{\bar k},A)\le \varepsilon$, then $A_{\bar k}$ is a Euclidean $\varepsilon$-outer approximation of $A$, i.e.,
\[
\delta_H(A_{\bar k},A;\|\cdot\|_2)\le \varepsilon.
\]

\item[(iii)] (Sufficient computable stopping test via the adaptive metric)
Assume additionally that the adaptive metrics satisfy the uniform spectral bounds $\varepsilon_0 I\preceq M_k\preceq (\varepsilon_0+1)I$ for all $k$ (as holds for the running-average metric $M_k=\varepsilon_0 I+\frac1k\sum_{i=1}^k u_i u_i^\top$; see Lemma~\ref{lem:spec_bounds}). If for some $\bar k$,
\[
\max_{v\in \operatorname{ext}(A_{\bar k})} \|z^v\|_{M_{\bar k}}
\ \le\ \sqrt{\varepsilon_0}\,\varepsilon,
\]
then $\mathrm{err}_2(A_{\bar k},A)\le \varepsilon$, which by definition of $\mathrm{err}_2$ is exactly the Euclidean Hausdorff bound $\delta_H(A_{\bar k},A;\|\cdot\|_2)\le \varepsilon$.
\end{enumerate}
\end{theorem}

\begin{proof}
(i) By assumption, $A_{0}\supseteq A$. If $A_{k}\supseteq A$ and $H_{k+1}\supseteq A$, then $A_{k+1}=A_{k}\cap H_{k+1}$ still contains $A$ and is clearly a subset of $A_{k}$.  (ii) \rev{Since $A_{\bar k}$ is a compact polytope containing $A$}, the Hausdorff error equals the maximum vertex distance, hence the condition $\mathrm{err}_2(A_{\bar k},A)\le \varepsilon$ is equivalent to $\delta_H(A_{\bar k},A;\|\cdot\|_2)\le \varepsilon$. (iii) For each vertex $v$, Lemma~\ref{lem:zk_proxy} gives $d_2(v,A)\le \frac{1}{\sqrt{\varepsilon_0}}\|z^v\|_{M_{\bar k}}$. Taking the maximum over $v\in \operatorname{ext}(A_{\bar k})$ yields $\mathrm{err}_2(A_{\bar k},A)\le \varepsilon$.
\end{proof}


The following lemma is the key mechanism ensuring progress: whenever the selected vertex lies outside $A$, the generated cut removes it from the outer approximation.

\begin{lemma}[The generated cut removes the selected vertex]\label{lem:cut_removes_vertex}
Fix $M\succ0$ and $v\in\mathbb{R}^q$, and consider the norm-minimization scalarization \eqref{P-v}. Let
$A:=P\cap S(\gamma)$ with $P=\Gamma(X)+C$.
Let $(x^\star,z^\star)$ be an optimal solution and set $y^\star:=v+z^\star\in A$.
Assume there exists a dual optimal pair $(w^\star,\lambda^\star)$ with $w^\star\in C^+$, $\lambda^\star\ge0$, and define the combined normal
\[
g^\star := w^\star-\lambda^\star \bar w,
\]
where $\bar w\in\operatorname{int} C^+$ is the fixed slice direction introduced in Section~\ref{sec:preliminaries}.
Then:
\begin{enumerate}
\item[(i)] (Supporting halfspace for the slice) The halfspace $H^\star:=\{y:\ (g^\star)^\top y\ge (g^\star)^\top y^\star\}$ contains $A$; equivalently, $(g^\star)^\top y \ge (g^\star)^\top y^\star$ for all $y\in A$.

\item[(ii)] (Strict separation of $v$ unless $v\in A$) If $z^\star\neq 0$ (equivalently, $v\notin A$), then
\[
(g^\star)^\top v \ <\ (g^\star)^\top y^\star,
\]
hence $v\notin H^\star$, so intersecting the current outer approximation with $H^\star$ removes the vertex $v$.
\end{enumerate}
\end{lemma}

\begin{proof}
(i) Let $y\in A$. Then $y=\Gamma(x)+c$ for some $x\in X$ and $c\in C$, and $\bar w^\top y\le\gamma$.
Since $w^\star\in C^+$, we have $(w^\star)^\top c\ge0$, hence $(w^\star)^\top y\ge (w^\star)^\top \Gamma(x)$.
Also, $\lambda^\star\ge0$ and $\bar w^\top y\le\gamma$ imply $\lambda^\star(\bar w^\top y-\gamma)\le 0$.
Therefore, for all $y\in A$,
\[
(w^\star)^\top y - \lambda^\star \bar w^\top y
\ \ge\
(w^\star)^\top \Gamma(x) - \lambda^\star \gamma.
\]
By strong duality \rev{(the scalarization is convex and satisfies Slater's condition: pick $x^\circ\in\operatorname{int}X$---nonempty by Assumption~\ref{ass:standing}---and $c^\circ\in\operatorname{int}C$ small enough that $\bar w^\top(\Gamma(x^\circ)+c^\circ)<\gamma$, possible because $\Gamma(X)\subseteq\operatorname{int}S(\gamma)$ by the slice construction of Section~\ref{sec:preliminaries}; then $z^\circ:=\Gamma(x^\circ)+c^\circ-v$ gives $\Gamma(x^\circ)-z^\circ-v=-c^\circ\in-\operatorname{int}C$ and $\bar w^\top(v+z^\circ)<\gamma$, so $(x^\circ,z^\circ)$ is strictly feasible)} and optimality of $(x^\star,z^\star)$ and $(w^\star,\lambda^\star)$, the standard saddle-point inequality yields
\[
(w^\star)^\top \Gamma(x) - \lambda^\star \gamma
\ \ge\
(w^\star)^\top \Gamma(x^\star) - \lambda^\star \gamma.
\]
\rev{By complementary slackness at the primal--dual optimum, $(w^\star)^\top c^\star=0$ and $\lambda^\star(\bar w^\top y^\star-\gamma)=0$, where $c^\star\in C$ is given by $y^\star=v+z^\star=\Gamma(x^\star)+c^\star$; hence $(g^\star)^\top y^\star=(w^\star)^\top\Gamma(x^\star)-\lambda^\star\gamma$.} Combining \rev{this with the two displays above, }we obtain
\[
(w^\star-\lambda^\star \bar w)^\top y \ \ge\ (w^\star-\lambda^\star \bar w)^\top y^\star,
\]
which proves~(i). (ii) Since $z$ is unconstrained, KKT stationarity in $z$ gives
\[
0 \in \partial \|z^\star\|_M \ +\ \{-w^\star+\lambda^\star \bar w\},
\quad\text{i.e.,}\quad
g^\star \in \partial \|z^\star\|_M.
\]
For $z^\star\neq 0$, the norm $\|z\|_M$ is differentiable and
\[
\partial \|z^\star\|_M = \left\{\frac{Mz^\star}{\|z^\star\|_M}\right\},
\]
hence $g^\star = Mz^\star/\|z^\star\|_M$ and therefore
\[
(g^\star)^\top z^\star
=
\frac{(z^\star)^\top M z^\star}{\|z^\star\|_M}
=
\|z^\star\|_M
>
0.
\]
Finally, since $y^\star=v+z^\star$,
\[
(g^\star)^\top y^\star
=
(g^\star)^\top v + (g^\star)^\top z^\star
=
(g^\star)^\top v + \|z^\star\|_M
>
(g^\star)^\top v,
\]
which proves strict separation and $v\notin H^\star$.
\end{proof}

We begin with a qualitative convergence guarantee, before turning to convergence \emph{rates}. The outer approximations $A_k$ converge to the slice $A$ in Euclidean Hausdorff distance. This requires only the uniform spectral bounds on $M_k$ (Lemma~\ref{lem:spec_bounds}) and the cut-removal property (Lemma~\ref{lem:cut_removes_vertex}), and is independent of \rev{any dispersion, curvature, or smoothness assumption}.

\begin{theorem}[Euclidean convergence under adaptive metrics]\label{thm:euclid_convergence_adaptive}
Let $A:=P\cap S(\gamma)$ be the compact slice of the upper image $P=\Gamma(X)+C$; \rev{the outer approximations $A_k$ are then compact polytopes containing $A$, their recession cone being $\{0\}$ because the slice direction satisfies $\bar w\in\operatorname{int}C^+$ (so $\bar w^\top d>0$ for all $d\in C\setminus\{0\}$).}
Assume that at each nonterminal iteration $k$ (i.e.\ $z^{v^k}\neq0$) the iteration produces \rev{the halfspace generated by the scalarization at the selected vertex $v^k$ (defined below), namely $H_{k+1}=\{y:(g^k)^\top y\ge(g^k)^\top y^k\}$ with combined cut-normal $g^k=M_k z^{v^k}/\|z^{v^k}\|_{M_k}$ (equivalently, $g^k=g_{k+1}$ in the subscript notation of the algorithm above) and boundary point $y^k=v^k+z^{v^k}$, so that $A\subseteq H_{k+1}$ (Lemma~\ref{lem:cut_removes_vertex})} and updates $A_{k+1}=A_{k}\cap H_{k+1}$.
Assume further that the scalarizations $(\mathrm P_{M_k}(v))$ are solved exactly---each returning a displacement $z^v$ that attains the true distance, $\|z^v\|_{M_k}=d_{M_k}(v,A)$, with no early termination---and that the matrices $M_k\succ0$ satisfy the uniform spectral bounds
\begin{equation}\label{eq:uniform_bounds_main}
m I \preceq M_k \preceq M I \qquad \forall k\ge 0
\end{equation}
for some constants $0<m\le M<\infty$.
At each iteration, select
\[
v^k \in \arg\max\{\ \|z^v\|_{M_k}\ :\ v\in \operatorname{ext}(A_{k})\ \},
\]
where $z^v$ denotes the optimal displacement, i.e., the $z$-part of the optimal solution of $(\mathrm P_{M_k}(v))$.

Then
\[
\delta_H(A_{k},A;\|\cdot\|_2)\to 0
\quad\text{as }k\to\infty.
\]
Equivalently,
\[
\max_{v\in \operatorname{ext}(A_{k})} d_2(v,A)\to 0.
\]
\end{theorem}

\begin{proof}
Throughout this proof, $\|\cdot\|$ (without subscript) denotes the Euclidean norm $\|\cdot\|_2$. By construction, $A\subseteq A_{0}$ and $A\subseteq H_{k+1}$ for all $k$, so $A\subseteq A_{k}$, and the sequence is nested and decreasing. For any vertex $v$ of $A_{k}$, let $d_k(v):=d_{M_k}(v,A)=\|z^v\|_{M_k}$ (distance identity from Lemma~\ref{lem:zk_proxy}). From \eqref{eq:uniform_bounds_main},
\begin{equation*}
\sqrt{m}\, d_2(v,A) \le d_k(v)\le \sqrt{M}\, d_2(v,A).
\end{equation*}
Let $E_k:=\delta_H(A_{k},A;\|\cdot\|_2)=\max_{v\in \operatorname{ext}(A_{k})} d_2(v,A)$ be the Euclidean Hausdorff error.
Since $v^k$ maximizes $d_k(\cdot)$ over vertices, $d_k(v^k)\ge \sqrt{m}\,E_k$, and hence
\[
d_2(v^k,A)\ge \frac{1}{\sqrt{M}}d_k(v^k)\ge \sqrt{\frac{m}{M}}\,E_k =: \theta_{\min}\, E_k.
\]
When $E_k>0$, we have $v^k\notin A$ and $\|z^{v^k}\|_{M_k}>0$. By Lemma~\ref{lem:cut_removes_vertex}(ii), $v^k\notin H_{k+1}$, so the update $A_{k+1}=A_{k}\cap H_{k+1}$ removes $v^k$.

It remains to show $E_k\to 0$ as $k\to\infty$. Since $A_{k+1}\subseteq A_{k}$, the sequence $E_k$ is non-increasing and bounded below by zero, hence converges to some $\ell\ge 0$. Suppose for contradiction that $\ell>0$. Then $E_k\ge \ell$ for all $k$, and therefore $d_2(v^k,A)\ge \theta_{\min} \ell$ and $\|z^{v^k}\|_{M_k}\ge \sqrt{m}\,\theta_{\min} \ell$ for all $k$. By the KKT stationarity condition (Lemma~\ref{lem:cut_removes_vertex}, part~(ii)), the cut-normal satisfies $g^k = M_k z^{v^k}/\|z^{v^k}\|_{M_k}$, and the Euclidean distance from $v^k$ to the cutting hyperplane $\partial H_{k+1}$ is
\[
d_2(v^k, \partial H_{k+1})
= \frac{\|z^{v^k}\|_{M_k}^2}{\|M_k z^{v^k}\|}
\ge \frac{m\,\|z^{v^k}\|^2}{M\,\|z^{v^k}\|}
= \frac{m}{M}\,\|z^{v^k}\|
\ge \frac{m}{M}\, d_2(v^k,A)
\ge \frac{m}{M}\,\theta_{\min} \ell
=: \delta > 0,
\]
using $\|z\|_{M_k}^2 \ge m\|z\|^2$, $\|M_k z\|\le M\|z\|$, and $\|z^{v^k}\| = \|y^k - v^k\| \ge d_2(v^k,A)$.

Since all $v^k$ lie in $A_{0}$ (compact), there exists a convergent subsequence $v^{k_j}\to v^*$. Choose $j_1 < j_2$ with $\|v^{k_{j_1}} - v^{k_{j_2}}\| < \delta/2$. The signed distance from $v^{k_{j_2}}$ to $\partial H_{k_{j_1}+1}$ satisfies
\[
\frac{\langle g^{k_{j_1}},\, v^{k_{j_2}} - y^{k_{j_1}}\rangle}{\|g^{k_{j_1}}\|}
\;\le\;
\frac{\langle g^{k_{j_1}},\, v^{k_{j_1}} - y^{k_{j_1}}\rangle}{\|g^{k_{j_1}}\|}
+ \|v^{k_{j_2}} - v^{k_{j_1}}\|
\;<\; -\delta + \tfrac{\delta}{2}
= -\tfrac{\delta}{2}
< 0,
\]
so $v^{k_{j_2}}\notin H_{k_{j_1}+1}$. But $v^{k_{j_2}}\in A_{k_{j_2}} \subseteq A_{k_{j_1}+1} \subseteq H_{k_{j_1}+1}$, a contradiction. Hence $\ell=0$.

For compact convex sets, $\delta_H(A_{k},A;\|\cdot\|_2) = \max_{v\in \operatorname{ext}(A_{k})} d_2(v,A)=E_k$, yielding the claim.
\end{proof}

\subsection[Metric Quality]{Metric Quality}

\rev{The constant in the convergence rate of the next subsection is governed by the conditioning of the adaptive metric, which we summarize in a single parameter.}

\begin{definition}[Metric quality parameter]\label{def:theta_k}
For $M_k\succ 0$, define the \emph{metric quality parameter}
\[
\theta_k := \sqrt{\frac{\lambda_{\min}(M_k)}{\lambda_{\max}(M_k)}} \in (0,1].
\]
Note that $\theta_k = 1/\kappa(M_k)^{1/2}$, where $\kappa(M_k)$ is the condition number of $M_k$. Thus $\theta_k\in(0,1]$ quantifies the conditioning of the metric: values close to $1$ indicate a well-conditioned, near-Euclidean $M_k$, while small $\theta_k$ reflects a large condition number and hence a larger constant in the convergence bounds (as for the fixed-norm case discussed after Theorem~\ref{thm:inner-product-rate}).
\end{definition}

\begin{lemma}[Selection quality in Euclidean distance]\label{lem:selection_quality}
Let $A$ be compact convex and let $M_k\succ0$. For any $v\in\mathbb{R}^q$,
\[
\sqrt{\lambda_{\min}(M_k)}\, d_2(v,A)\ \le\ d_{M_k}(v,A)\ \le\ \sqrt{\lambda_{\max}(M_k)}\, d_2(v,A).
\]
If $v^k\in\arg\max_{v} d_{M_k}(v,A)$, then
\[
d_2(v^k,A)\ \ge\ \theta_k\, E_k,
\]
where $E_k := \delta_H(A_k,A;\|\cdot\|_2) = \max_{v\in\operatorname{ext}(A_k)} d_2(v,A)$ is the Euclidean Hausdorff error.

\end{lemma}

\begin{proof}
The first inequalities follow from Rayleigh quotient bounds: $\lambda_{\min}(M_k)\|d\|_2^2 \le d^\top M_k d \le \lambda_{\max}(M_k)\|d\|_2^2$.
Taking infimum over $y\in A$ yields the distance bounds. If $v^k$ maximizes $d_{M_k}(\cdot,A)$, then
\[
d_{M_k}(v^k,A) \ge \max_v d_{M_k}(v,A) \ge \sqrt{\lambda_{\min}(M_k)}\,E_k,
\]
and also $d_{M_k}(v^k,A)\le \sqrt{\lambda_{\max}(M_k)}\, d_2(v^k,A)$, giving $d_2(v^k,A)\ge \theta_k E_k$.
\end{proof}

\begin{revblock}
\begin{remark}[Empirical dispersion of the cut-normals]\label{rem:dispersion}
The metric quality $\theta_k$ is controlled by how well the cut-normals spread across directions. Writing $\Sigma_k:=\frac1k\sum_{i=1}^k u_iu_i^\top$ for the empirical second-moment matrix of the normalized cut-normals $u_i$ (so that $M_k=\varepsilon_0 I+\Sigma_k$), we have $\theta_k\ge\sqrt{(\varepsilon_0+\lambda_{\min}(\Sigma_k))/(\varepsilon_0+1)}$; a floor $\lambda_{\min}(\Sigma_k)\ge\eta>0$ therefore improves the metric quality to $\theta_k\ge\sqrt{(\varepsilon_0+\eta)/(\varepsilon_0+1)}$. A positive smallest eigenvalue $\lambda_{\min}(\Sigma_k)$ is observed \emph{at termination} in the four main $q\le3$ runs (Section~\ref{sec:experiments}), where it serves as a convenient diagnostic of metric conditioning; we do not claim a uniform-in-$k$ floor $\eta$ (indeed $\Sigma_1=u_1u_1^\top$ is rank one, so $\lambda_{\min}(\Sigma_1)=0$). We do not require it: Theorem~\ref{thm:rate} establishes the improved rate with no dispersion hypothesis. Whether optimal-rate convergence in turn \emph{forces} the cut-normals to disperse is an interesting converse question, which we leave open.
\end{remark}
\end{revblock}

\subsection{Convergence Rate}

For context we first recall the $H(r,A)$-sequence framework of~\cite{kamenev1992class,lotov2004interactive}, which yields the generic rate $O(k^{1/(1-q)})$; Theorem~\ref{thm:rate} then improves on this exponent by adapting the Euclidean packing argument of \cite{ararat_convergence_2024} directly to the iteration-dependent metric, rather than through this framework.

\begin{definition}[$H(r,A)$-sequence of cutting]\label{def:H_sequence}
Let $A\subset\mathbb{R}^q$ be a nonempty convex compact set and let $r>0$. A sequence $(A_k)_{k\ge 0}$ of polytopes in $\mathbb{R}^q$ is called an \emph{$H(r,A)$-sequence of cutting} if:
\begin{enumerate}
\item $A_0 = \bigcap_{i=1}^{I} \mathcal{H}(\omega_i, A)$ for some finite collection $\omega_1,\ldots,\omega_I \in \mathbb{R}^q\setminus\{0\}$;
\item $A_k \supseteq A$ for all $k\ge 0$;
\item $A_{k+1} = A_k \cap \mathcal{H}(w_k, A)$ for some $w_k \in \mathbb{R}^q\setminus\{0\}$;
\item $\delta_H(A_k, A_{k+1}) \ge r\cdot \delta_H(A_k, A)$ for all $k\ge 0$.
\end{enumerate}
The first condition ensures that $A_0$ is a polytope containing $A$, while the fourth condition ensures that each cut removes at least a fraction $r$ of the current approximation error. For such sequences, the general rate theorem~\cite{kamenev1992class,lotov2004interactive} establishes that $\delta_H(A_k,A) = O(k^{1/(1-q)})$.
\end{definition}

\begin{revblock}
\begin{theorem}[Improved convergence rate of the adaptive-metric algorithm]\label{thm:rate}
Let $A:=P\cap S(\gamma)\subset\mathbb{R}^q$ ($q\ge2$) be nonempty, convex, and compact with circumradius $R:=R(A)>0$, and let $(A_k)_{k\ge0}$ be generated by the adaptive-metric algorithm under the hypotheses of Theorem~\ref{thm:euclid_convergence_adaptive} (in particular the scalarization-generated cuts $H_{k+1}$, and the compactness of the $A_k$, established there): exact scalarizations $(\mathrm P_{M_k}(v))$ (each returning $z^v$ with $\|z^v\|_{M_k}=d_{M_k}(v,A)$), uniform spectral bounds $mI\preceq M_k\preceq MI$ ($0<m\le M<\infty$; by Lemma~\ref{lem:spec_bounds} one may take $m=\varepsilon_0$, $M=\varepsilon_0+1$), and the farthest-vertex rule~\eqref{eq:vertex_rule}. Write $E_k:=\delta_H(A_k,A;\|\cdot\|_2)$ and $\theta:=\sqrt{m/M}$. Then either the algorithm terminates with $A_{\bar k}=A$ after finitely many steps, or for every $\varepsilon>0$ there exists $K\in\mathbb{N}$ such that
\[
E_k\ \le\ (1+\varepsilon)\,\frac{1}{\theta}\,\bar\lambda(A)\,k^{\frac{2}{1-q}}
\quad\text{for all }k\ge K,
\qquad
\bar\lambda(A):=16R\Bigl(\frac{q\pi_q}{\pi_{q-1}}\Bigr)^{\frac{2}{q-1}},
\]
where $\pi_q$ is the volume of the Euclidean unit ball in $\mathbb{R}^q$. In particular $E_k=O(k^{2/(1-q)})$---the improved exponent, previously established only for fixed inner-product norms (Theorem~\ref{thm:inner-product-rate})---with no strict-convexity or smoothness hypothesis on $A$.
\end{theorem}

\begin{proof}
Throughout, $\|\cdot\|:=\|\cdot\|_2$ and $\varrho>0$ is a free parameter, optimized only at the end. If $E_{\bar k}=0$ for some $\bar k$ the claim is trivial, so assume $E_k>0$ for all $k$. For each $k$ write $z^k:=z^{v^k}$, $y^k:=v^k+z^k\in A$, and $g^k:=M_kz^k/\|z^k\|_{M_k}$ for the combined cut-normal (Lemma~\ref{lem:cut_removes_vertex}), so that $H_{k+1}=\{y:(g^k)^\top y\ge(g^k)^\top y^k\}\supseteq A$ (the cut added at iteration $k$, matching Theorem~\ref{thm:euclid_convergence_adaptive}), and the update reads $A_{k+1}=A_k\cap H_{k+1}$; put $n^k:=g^k/\|g^k\|\in\mathbb{S}^{q-1}$, $s_k:=\|z^k\|_{M_k}=d_{M_k}(v^k,A)$, and $h_k:=(n^k)^\top z^k=s_k^2/\|M_kz^k\|$.

\emph{Step 1 (cut depth and displacement).} Since $v^k$ maximizes $d_{M_k}(\cdot,A)$ and $\sqrt m\,d_2(v,A)\le d_{M_k}(v,A)\le\sqrt M\,d_2(v,A)$ (Lemma~\ref{lem:selection_quality}), we have $s_k=\max_v d_{M_k}(v,A)\ge\sqrt m\,E_k$ and $s_k=d_{M_k}(v^k,A)\le\sqrt M\,E_k$, whence $\|z^k\|\le s_k/\sqrt m\le\sqrt{M/m}\,E_k=E_k/\theta$. As $M_k\preceq MI$ gives $M_k^2\preceq M M_k$, we obtain $\|M_kz^k\|^2=(z^k)^\top M_k^2z^k\le M s_k^2$, so
\[
h_k=\frac{s_k^2}{\|M_kz^k\|}\ \ge\ \frac{s_k}{\sqrt M}\ \ge\ \sqrt{\tfrac mM}\,E_k=\theta E_k.
\]

\emph{Step 2 (separation of lifted points).} Put $b^i:=y^i-\varrho\,n^i$. Since $H_{i+1}\supseteq A$ supports $A$ at $y^i\in A$ and $\|n^i\|=1$, each $b^i$ lies on $\operatorname{bd}(A+\overline B(0,\varrho))$. We claim that for every $i<k$,
\begin{equation}\label{eq:adaptive_separation}
\|b^k-b^i\|\ \ge\ \min\bigl\{\,\varrho\sqrt2,\ \sqrt{\varrho h_k}-\|z^k\|\,\bigr\}.
\end{equation}
Indeed $y^i\in A\subseteq H_{k+1}$ and $v^k\in A_k\subseteq H_{i+1}$ (as $A_k=A_0\cap H_1\cap\dots\cap H_k\subseteq H_{i+1}$ for $i<k$), so both incidences $y^i\in H_{k+1}$ and $v^k\in H_{i+1}$ hold, and $y^i$ is the $n^i$-support point of $A$. If $(n^k)^\top n^i\le0$, expanding $\|b^k-b^i\|^2$ and discarding the two nonnegative cross terms---$(n^k)^\top(y^i-y^k)\ge0$ and $(n^i)^\top(y^k-y^i)\ge0$---gives $\|b^k-b^i\|^2\ge\varrho^2\|n^k-n^i\|^2\ge2\varrho^2$. If $(n^k)^\top n^i>0$, apply \cite[Lemma~7.4]{ararat_convergence_2024} with $(y,y',w,w')=(y^i,v^k,n^i,n^k)$ to get $\|(v^k-\varrho n^k)-(y^i-\varrho n^i)\|^2\ge\varrho\,d\bigl(y^i,\operatorname{bd}H(n^k,v^k)\bigr)=\varrho\,(n^k)^\top(y^i-v^k)\ge\varrho h_k$, using $(n^k)^\top(y^i-v^k)=(n^k)^\top(y^i-y^k)+h_k\ge h_k$; combining with the triangle inequality $\|b^k-b^i\|\ge\|(v^k-\varrho n^k)-(y^i-\varrho n^i)\|-\|v^k-y^k\|$ and $\|v^k-y^k\|=\|z^k\|$ yields~\eqref{eq:adaptive_separation}. (This is the adaptive analogue of \cite[Lemma~7.5]{ararat_convergence_2024}; the Euclidean alignment used there, $\|v-y\|=n^\top(y-v)$, fails for the moving metric, so $\|z^k\|$ replaces $h_k$ in the second term.)

\emph{Step 3 (packing bound and inversion).} Fix $\rho\in(0,1)$. By Step 1, whenever $E_k\le\rho^2\varrho\theta^3$ the second term in~\eqref{eq:adaptive_separation} obeys $\sqrt{\varrho h_k}-\|z^k\|\ge\sqrt{\varrho\theta E_k}-E_k/\theta\ge(1-\rho)\sqrt{\varrho\theta E_k}$; since $E_k\to0$ (Theorem~\ref{thm:euclid_convergence_adaptive}) there is $N$ such that this holds and $(1-\rho)\sqrt{\varrho\theta E_k}<\varrho\sqrt2$ for all $k\ge N$. As $E_k$ is non-increasing, for $N\le i<j\le k$ inequality~\eqref{eq:adaptive_separation} gives $\|b^i-b^j\|\ge(1-\rho)\sqrt{\varrho\theta E_k}=:2\tau_k$, so $\{b^i\}_{N\le i\le k}$ is the base of a $\tau_k$-packing (\cite[Def.~7.8]{ararat_convergence_2024}) on $\operatorname{bd}(A+\overline B(0,\varrho))$, whose circumradius is at most $R+\varrho$. By \cite[Lemma~7.10]{ararat_convergence_2024},
\[
k-N+1\ \le\ \frac{q\pi_q}{\pi_{q-1}}\Bigl(1+\frac{4(R+\varrho)^2}{(1-\rho)^2\,\varrho\theta\,E_k}\Bigr)^{(q-1)/2}.
\]
Solving this inequality for $E_k$ gives, for all $k\ge N$,
\[
E_k\ \le\ \frac{4(R+\varrho)^2}{(1-\rho)^2\,\varrho\,\theta}\left[\Bigl(\frac{k-N+1}{q\pi_q/\pi_{q-1}}\Bigr)^{2/(q-1)}-1\right]^{-1}.
\]
Letting $k\to\infty$ and then $\rho\to0$, and finally minimizing $(R+\varrho)^2/\varrho$ over $\varrho>0$ at $\varrho=R$ (where $(R+\varrho)^2/\varrho=4R$), we obtain: for every $\varepsilon>0$ there is $K$ with $E_k\le(1+\varepsilon)\,\theta^{-1}\bar\lambda(A)\,k^{2/(1-q)}$ for all $k\ge K$, where $\bar\lambda(A)=16R\,(q\pi_q/\pi_{q-1})^{2/(q-1)}$.
\end{proof}
\end{revblock}

\begin{revblock}
\begin{remark}[Role of uniform conditioning in the proof]\label{rem:r_positive}
The positive lower bound $\theta=\sqrt{m/M}>0$ underlying Theorem~\ref{thm:rate} is precisely what the $\varepsilon_0$-regularization provides. By Lemma~\ref{lem:spec_bounds}, $\varepsilon_0 I\preceq M_k\preceq(\varepsilon_0+1)I$ for every $k$, so one may take $m=\varepsilon_0$, $M=\varepsilon_0+1$, giving
\[
\theta\ \ge\ \sqrt{\frac{\varepsilon_0}{\varepsilon_0+1}}\ >\ 0
\]
for every $k$, with no dispersion or geometric hypothesis on $A$. This regularization is not a mere numerical safeguard; it is what the proof uses: the packing separation in Step~2 requires a $k$-independent positive lower bound on $\theta$ (equivalently, a uniform upper bound on $\kappa(M_k)$), and the $\varepsilon_0$-regularization supplies it. Whether the improved exponent can persist for an \emph{unregularized} moving metric whose condition number $\kappa(M_k)$ is unbounded is a question our argument leaves open. The fixed $\varepsilon_0$-regularization is one simple mechanism that supplies the required uniform conditioning---indeed it caps $\kappa(M_k)\le(\varepsilon_0+1)/\varepsilon_0$---and thereby secures the improved rate. The constant $\theta^{-1}\bar\lambda(A)=\sqrt{(\varepsilon_0+1)/\varepsilon_0}\,\bar\lambda(A)$ degrades gracefully as $\varepsilon_0\to0$, quantifying the trade-off between adaptivity (small $\varepsilon_0$) and conditioning (large $\varepsilon_0$).
\end{remark}
\end{revblock}


\section{Numerical Experiments}\label{sec:experiments}

\subsection{Test Problems}

We consider three test problems from the literature.

\paragraph{Example 1 (Ball).}
From \cite{ararat_convergence_2024}, the problem
\begin{equation*}
\min_{x \in X} x \quad \text{w.r.t.} \leq_{\mathbb{R}^q_+}, \qquad
X = \{x \in \mathbb{R}^q : \|x - e\|_2 \leq 1\},
\end{equation*}
where $e = (1,\ldots,1)^\top$. The (weakly) efficient frontier is the portion of the unit sphere centered at $e$ that lies in $\{e\}-\mathbb{R}^q_+$ (equivalently, the part of $\partial X$ whose outer normals lie in $-\mathbb{R}^q_+$). We test $q \in \{2, 3, 4\}$.

\paragraph{Example 2.}
From \cite{ararat_convergence_2024}, the problem
\begin{equation*}
\min_{x \in X} \begin{pmatrix} \|x - a_1\|^2 \\ \|x - a_2\|^2 \\ \|x - a_3\|^2 \end{pmatrix} \quad \text{w.r.t.} \leq_{\mathbb{R}^3_+}, \qquad
X = \{x \in \mathbb{R}^2 : x_1 + 2x_2 \leq 10,\ 0 \leq x_1 \leq 10,\ 0 \leq x_2 \leq 4\},
\end{equation*}
where $a_1 = (1,1)^\top$, $a_2 = (2,3)^\top$, $a_3 = (4,2)^\top$.
This is a tri-objective problem ($q=3$) with a two-dimensional decision space ($n=2$) and a polyhedral constraint set. The objectives measure squared distances to three reference points, producing a curved Pareto front.

\paragraph{Example 3 (Jahn 11.4).}
From \cite{jahn_vector_2011}, the problem
\begin{equation*}
\min_{x \in X} \begin{pmatrix} -x_1 \\ x_1 + x_2^2 \end{pmatrix} \quad \text{w.r.t.} \leq_{\mathbb{R}^2_+}, \qquad
X = \{x : x_1^2 - x_2 \leq 0,\ x_1 + 2x_2 \leq 3\}.
\end{equation*}
This has a parabolic Pareto front with different curvature properties than Example~1.

\subsection{Experimental Setup}

All experiments were implemented in MATLAB R2025b using CVX~\cite{cvx,gb08} for convex optimization and BENSOLVE Tools 1.3~\cite{bensolve} for polyhedral computations. \rev{For the reported full-enumeration runs ($q\le 3$), the vertices of each outer approximation are computed by a deterministic polar-duality/convex-hull backend (MATLAB's \texttt{convhulln}, i.e.\ the Quickhull algorithm~\cite{Barber1996}), guarded by a fail-closed check that rejects any returned vertex violating the current halfspace representation; we adopted this backend after the default BENSOLVE routine stalled on one instance and, on another, returned a vertex that the feasibility check flagged as infeasible. We additionally cross-checked the enumerated vertex set against the halfspace representation at termination over a finite bank of fixed support directions ($8$ for $q=2$, $26$ for $q=3$); this is a terminal consistency check on the representative runs, not a per-iteration or general-completeness certificate. The periodic full-enumeration passes of the $q=4$ hybrid strategy use the guarded default BENSOLVE backend.} All main comparison runs use regularization parameter $\varepsilon_0 = 0.1$ (the sensitivity study of Section~\ref{sec:eps0-sensitivity} varies it). This value balances adaptation against conditioning. \rev{By Lemma~\ref{lem:spec_bounds}, the metric quality satisfies $\theta_k\ge\sqrt{\varepsilon_0/(\varepsilon_0+1)}$ (improving to $\sqrt{(\varepsilon_0+\eta)/(\varepsilon_0+1)}$ when the cut-normals disperse with floor $\eta$; see Remark~\ref{rem:dispersion})}, so a smaller $\varepsilon_0$ lets $M_k$ track the empirical second-moment matrix $\Sigma_k$ more closely (stronger adaptation, but poorer conditioning), whereas a larger $\varepsilon_0$ pulls $M_k$ toward the Euclidean metric (better conditioned but less adaptive); the choice $\varepsilon_0=0.1$ kept \rev{the terminal }$\theta_k$ in \rev{$[0.46,0.67]$} \rev{for the four main $q\le 3$ runs, with $\theta_k$ never below its floor $\sqrt{\varepsilon_0/(\varepsilon_0+1)}\approx 0.30$ (attained at $k=1$, where $\Sigma_1$ is rank one)} (Table~\ref{tab:metric-quality}). \rev{The cut is anchored at the certified boundary point $y_k = v_k + z_k$ (Lemma~\ref{lem:cut_removes_vertex}); to avoid cancellation as $z_k\to 0$ we form its offset in the algebraically equivalent stable form $g_{k+1}^\top y_k = g_{k+1}^\top v_k + \|z_k\|_{M_k}$ rather than by first summing $v_k+z_k$.} \rev{Convergence rates are estimated by ordinary least squares of the plotted log-error against $\log_{10}j$ over the full run, where $j=k+1$ is the one-based evaluation index (the number of approximations evaluated; since $j\sim k$ asymptotically, the theoretical exponent $2/(1-q)$ is unchanged). At evaluation $j$, the plotted quantities are $\delta_H^{j-1}$ for the fixed Euclidean norm and the certified Euclidean upper proxy $\widehat E_{j-1}=\max_{v\in \operatorname{ext}(A_{j-1})}\|z^v\|_2\in[E_{j-1},\theta_{j-1}^{-1}E_{j-1}]$ (Corollary~\ref{cor:Ehat}) for the adaptive metric.}

Table~\ref{tab:parameters} summarizes the problem-specific parameters. For $q \leq 3$, we use full vertex enumeration; for $q = 4$, we employ the hybrid and linear-program (LP) probe strategies described in Section~\ref{sec:vertex-strategies}.

\begin{table}[htbp]
\centering
\caption{Experimental parameters for each test problem.}
\label{tab:parameters}
\begin{tabular}{llccl}
\toprule
Example & $q$ & Tolerance $\varepsilon$ & Max evaluations & Strategy \\
\midrule
1 (Ball) & 2 & $10^{-5}$ & 500 & Full \\
1 (Ball) & 3 & $0.01$    & 500 & Full \\
1 (Ball) & 4 & $0.0496$  & 500 & Hybrid / LP$^\dagger$ \\
2        & 3 & $0.02$    & 500 & Full \\
3 (Jahn) & 2 & $10^{-3}$ & 500 & Full \\
\bottomrule
\multicolumn{5}{l}{\footnotesize $^\dagger$See Section~\ref{sec:vertex-strategies} for strategy descriptions.}
\end{tabular}
\end{table}

\subsection[Convergence Rates]{Convergence Rates}\label{sec:convergence-results}

\rev{Figure~\ref{fig:convergence} shows, on a log-log scale for all four test configurations against the evaluation index $j=k+1$, the exact Euclidean Hausdorff error $\delta_H^{j-1} := \delta_H(A_{j-1}, A; \|\cdot\|_2)$ for the fixed Euclidean norm and the certified Euclidean upper proxy $\widehat E_{j-1}$ for the adaptive metric.} Table~\ref{tab:convergence} reports the \rev{evaluation counts and fitted slopes}.

\begin{figure}[htbp]
\centering
\includegraphics[width=\textwidth]{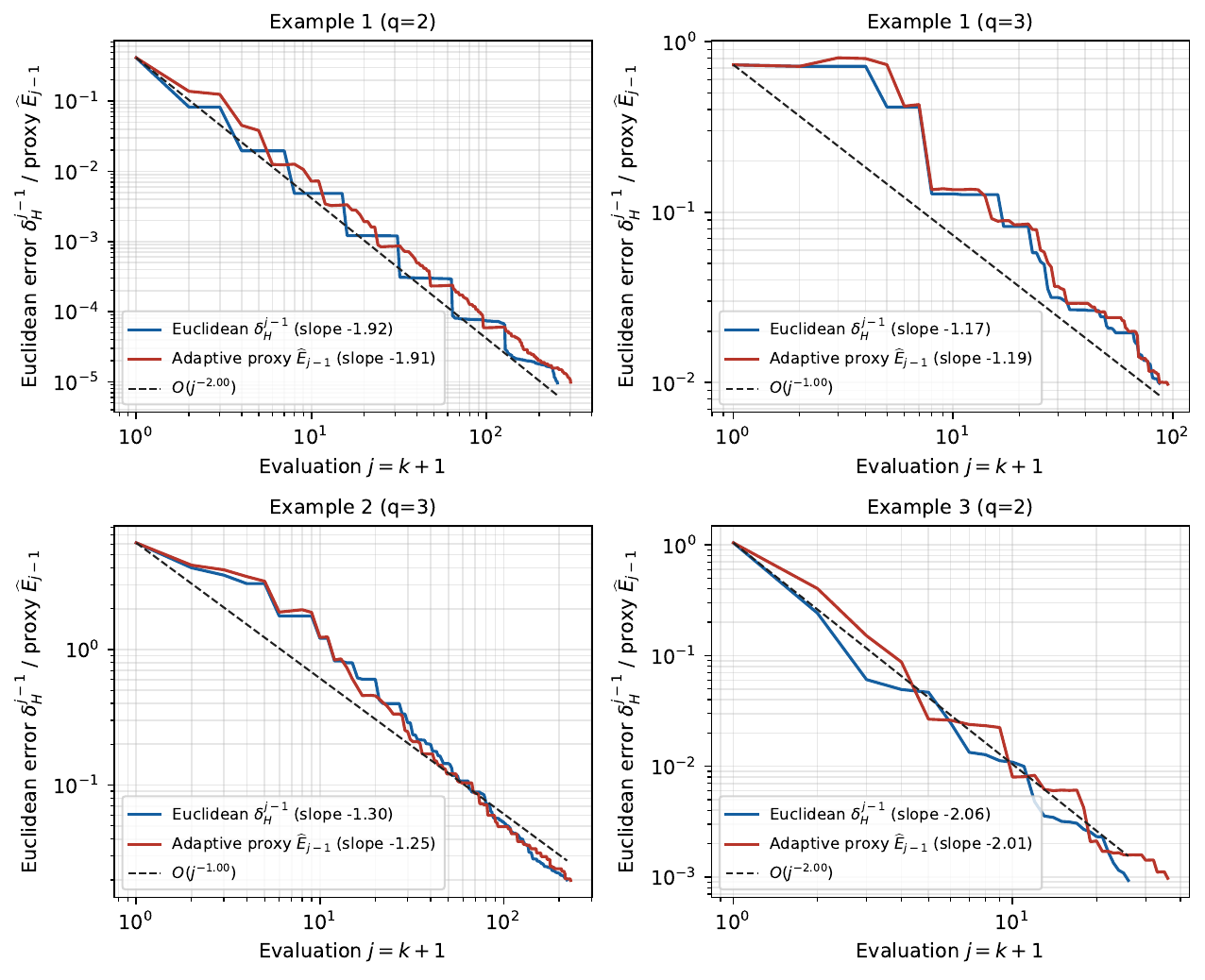}
\caption{Convergence of the Euclidean Hausdorff error $\delta_H^{j-1}$ (blue) and the adaptive certified Euclidean upper proxy $\widehat E_{j-1}$ (red), plotted against the evaluation index $j=k+1$. Dashed lines show the theoretical rate $O(j^{2/(1-q)})$ (Theorem~\ref{thm:inner-product-rate} for the fixed norm, Theorem~\ref{thm:rate} for the adaptive metric). Fitted slopes (shown in the legend) are ordinary-least-squares fits over the full run.}
\label{fig:convergence}
\end{figure}

\begin{table}[htbp]
\centering
\caption{Convergence of the Euclidean and adaptive metrics. The fitted slope is an ordinary-least-squares fit against $\log_{10}j$ ($j=k+1$, the one-based evaluation index) over the full run of the plotted log-error---$\log_{10}\delta_H^{j-1}$ for the Euclidean rows and $\log_{10}\widehat E_{j-1}$ for the adaptive rows, where $\widehat E_k\in[E_k,\theta_k^{-1}E_k]$ is the certified Euclidean upper proxy (Corollary~\ref{cor:Ehat}); the theoretical rate is $2/(1-q)$ (Theorem~\ref{thm:inner-product-rate} for the fixed norm, Theorem~\ref{thm:rate} for the adaptive metric).}
\label{tab:convergence}
\begin{tabular}{llcccc}
\toprule
Example & $q$ & Norm & Evaluations & Fitted slope & Theory \\
\midrule
1 (Ball) & 2 & Euclidean & 254 & $-1.92$ & $-2.0$  \\
         &   & Adaptive  & 302 & $-1.91$ &         \\
\addlinespace
1 (Ball) & 3 & Euclidean & 87  & $-1.17$ & $-1.0$  \\
         &   & Adaptive  & 95  & $-1.19$ &         \\
\addlinespace
2        & 3 & Euclidean & 221 & $-1.30$ & $-1.0$  \\
         &   & Adaptive  & 231 & $-1.25$ &         \\
\addlinespace
3 (Jahn) & 2 & Euclidean & 26  & $-2.06$ & $-2.0$  \\
         &   & Adaptive  & 36  & $-2.01$ &         \\
\bottomrule
\end{tabular}
\end{table}

\begin{revblock}
The certified runs show no evaluation advantage for the adaptive metric: on every instance it terminates in a comparable---in fact slightly larger---number of evaluations than the fixed Euclidean norm (Table~\ref{tab:convergence}). Evaluation-count reductions observed with an earlier implementation traced to an invalid cut anchor and disappear under the certified anchoring described in the experimental setup. The benefit of the adaptive metric is accordingly not a reduced evaluation count; it is the guarantee of Theorem~\ref{thm:rate} that the improved rate is attained with a problem-adapted metric that remains uniformly well-conditioned throughout.

The fitted slopes are consistent with the theoretical rate. For the bi-objective instances the slopes are close to the predicted $-2$ (Example~1: $-1.92$ Euclidean, $-1.91$ adaptive; Example~3: $-2.06$ and $-2.01$), and for the tri-objective instances they are at or slightly steeper than the predicted $-1$ (Example~1: $-1.17$ and $-1.19$; Example~2: $-1.30$ and $-1.25$), the latter consistent with the theory providing an upper bound on the asymptotic error. Figure~\ref{fig:convergence} shows both error curves tracking the $O(j^{2/(1-q)})$ reference lines. For the adaptive metric this is the behavior established by Theorem~\ref{thm:rate}, which yields the same improved rate with a constant inflated by $\theta^{-1}=\sqrt{(\varepsilon_0+1)/\varepsilon_0}$ relative to the fixed-norm bound.
\end{revblock}

\begin{revblock}
Finally, we comment on monotonicity. Because the outer approximations are nested ($A_{k+1} \subseteq A_k$), the exact Euclidean Hausdorff error $\delta_H^k$ is nonincreasing in $k$, as the Euclidean curves in Figure~\ref{fig:convergence} confirm. Two plotted quantities can nevertheless rise transiently: the certified upper proxy $\widehat E_k$ in the adaptive runs, because the metric $M_k$ changes from one iteration to the next, and the candidate-set residuals in Figure~\ref{fig:q4-comparison}, because the candidate vertex set is incomplete between enumeration passes. This accounts for the small upward steps visible in those curves. For the full-enumeration runs convergence to zero is guaranteed (Theorem~\ref{thm:euclid_convergence_adaptive}); the $q=4$ candidate-set curves use the probe-based strategies of Section~\ref{sec:vertex-strategies}, for which we do not establish the fixed-fraction property and therefore claim no convergence or rate guarantee.
\end{revblock}

\subsection{Adaptive Metric Performance}\label{sec:metric-performance}

Figure~\ref{fig:metric-evolution} illustrates the evolution of the adaptive metric for Example~1 with $q=2$. The metric quality parameter $\theta_k = \sqrt{\lambda_{\min}(M_k)/\lambda_{\max}(M_k)}$ stabilizes near $0.5$ after an initial transient. The eigenvalues of the empirical second-moment matrix $\Sigma_k = (1/k)\sum_{i=1}^k u_i u_i^\top$ separate early and remain bounded: $\lambda_{\min}(\Sigma_k) \approx 0.17$ and $\lambda_{\max}(\Sigma_k) \approx 0.83$ at convergence, so $\lambda_{\min}(\Sigma_k) > 0$ for large $k$\rev{ (Remark~\ref{rem:dispersion})}; the minimum eigenvalue is small during the initial transient, when few cut-normals have yet accumulated.

\begin{figure}[htbp]
\centering
\includegraphics[width=\textwidth]{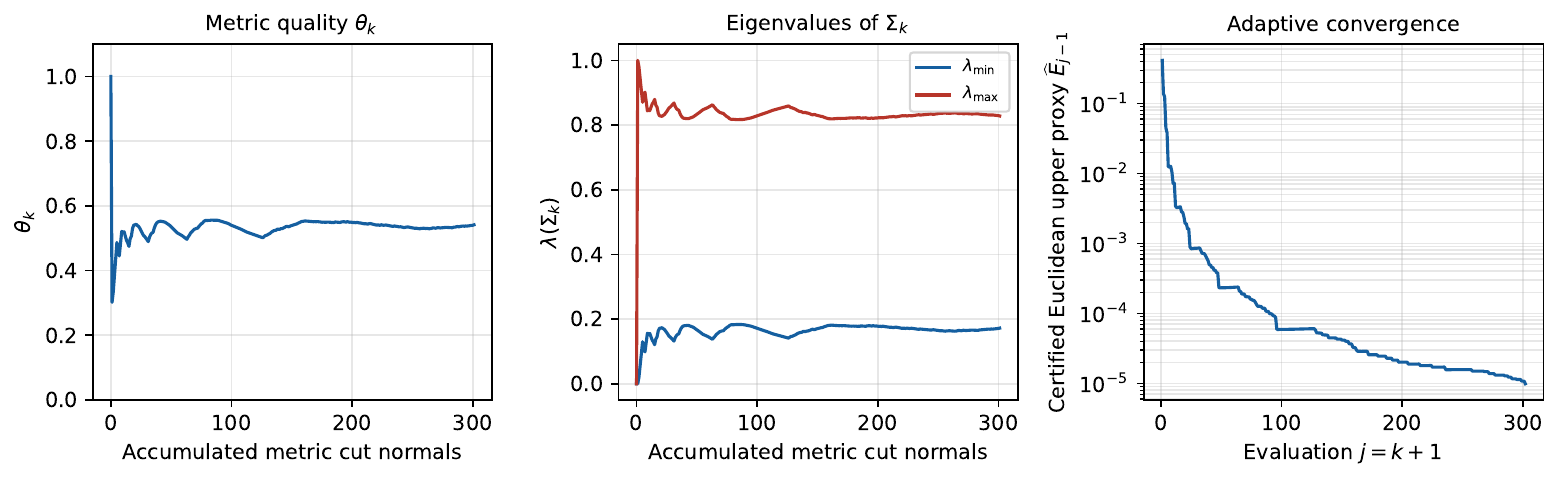}
\caption{Adaptive metric evolution for Example~1 ($q=2$). Left: metric quality $\theta_k$ stabilizes near $0.5$. Center: eigenvalues of $\Sigma_k$; $\lambda_{\min}(\Sigma_k)$ is small during the initial transient and settles at a positive value \rev{(for large $k$)}. Right: \rev{the certified Euclidean upper proxy $\widehat E_{j-1}$ (against evaluation $j=k+1$)} on a logarithmic scale.}
\label{fig:metric-evolution}
\end{figure}

\begin{table}[htbp]
\centering
\caption{Final metric quality parameter $\theta_k = \sqrt{\lambda_{\min}(M_k)/\lambda_{\max}(M_k)}$ at convergence.}
\label{tab:metric-quality}
\begin{tabular}{llcc}
\toprule
Example & $q$ & $\theta_k$ & Evaluations \\
\midrule
1 (Ball) & 2 & 0.54 & 302 \\
1 (Ball) & 3 & 0.65 & 95  \\
2        & 3 & 0.67 & 231 \\
3 (Jahn) & 2 & 0.46 & 36  \\
\bottomrule
\end{tabular}
\end{table}

Table~\ref{tab:metric-quality} reports the final metric quality $\theta_k$ in the four main $q\le3$ runs; the \rev{terminal }values $\theta_k \in [0.46, 0.67]$ \rev{(the four main $q\le 3$ runs) }indicate well-conditioned metrics\rev{, with $\theta_k$ never below its floor $\sqrt{\varepsilon_0/(\varepsilon_0+1)}\approx 0.30$}, consistent with the spectral bounds $\varepsilon_0 I \preceq M_k \preceq (\varepsilon_0 + 1)I$ from Lemma~\ref{lem:spec_bounds}. \rev{This empirical observation is consistent with Remark~\ref{rem:dispersion}: in these runs the cut-normals spread across the directions actually generated, keeping the terminal $\lambda_{\min}(\Sigma_k)$ positive---in the range $[0.107,0.245]$ across the four runs---and hence $\theta_k$ well above its worst-case floor $\sqrt{\varepsilon_0/(\varepsilon_0+1)}$. We do not attribute this spreading to front curvature; whether optimal-rate convergence \emph{forces} it remains open.}

\rev{The weakest conditioning occurs in Example~3 ($\theta_k = 0.46$), which is also the shortest run ($36$ evaluations, hence $35$ accumulated cut-normals at the terminal evaluation) and so gives the metric the least data to adapt to. We record this as an observation and do not claim that the parabolic front's lower curvature is its cause.}

\begin{revblock}
\subsection{Sensitivity to the Regularization Parameter}\label{sec:eps0-sensitivity}

To assess the role of the regularization $\varepsilon_0$, we sweep $\varepsilon_0\in\{0.03,0.1,0.3\}$ on the two full-enumeration $q=3$ instances, running each to a fixed budget of $50$ evaluations (medians over repeated timed runs; the reported quantities agreed across repeats to within tolerance). Table~\ref{tab:eps0-sensitivity} makes precise the two competing effects of Lemma~\ref{lem:spec_bounds}: a larger $\varepsilon_0$ raises the metric-quality floor $\theta_{\min}=\sqrt{\varepsilon_0/(\varepsilon_0+1)}$ (better conditioning), while a smaller $\varepsilon_0$ lets $M_k$ track the empirical second-moment matrix more closely (stronger adaptation). The recorded $\theta_{\min}$ equals this floor exactly ($0.171$, $0.302$, $0.480$ for $\varepsilon_0=0.03,0.1,0.3$). The default $\varepsilon_0=0.1$ attains the smallest error after the fixed budget on both instances, and the variation across the three tested values is modest. We read this as limited empirical support---from a two-instance, three-value, $50$-evaluation sweep---that the method is not overly sensitive to $\varepsilon_0$ near the default, rather than a general tuning-insensitivity result.

\begin{table}[htbp]
\centering
\caption{Sensitivity to the regularization $\varepsilon_0$ on the two $q=3$ full-enumeration instances, at a fixed budget of $50$ evaluations. $\widehat E_{49}$ is the certified Euclidean upper proxy at the terminal (50th) evaluation; $\theta_{\min}=\sqrt{\varepsilon_0/(\varepsilon_0+1)}$ is the metric-quality floor.}
\label{tab:eps0-sensitivity}
\begin{tabular}{llccc}
\toprule
Problem & $\varepsilon_0$ & Error $\widehat E_{49}$ & $\theta_{\text{final}}$ & $\theta_{\min}$ \\
\midrule
1 (Ball, $q=3$) & 0.03 & 0.0258 & 0.602 & 0.171 \\
                & 0.10 & \textbf{0.0241} & 0.711 & 0.302 \\
                & 0.30 & 0.0258 & 0.833 & 0.480 \\
\addlinespace
2 ($q=3$)       & 0.03 & 0.1358 & 0.681 & 0.171 \\
                & 0.10 & \textbf{0.1210} & 0.751 & 0.302 \\
                & 0.30 & 0.1225 & 0.839 & 0.480 \\
\bottomrule
\end{tabular}
\end{table}
\end{revblock}

\subsection{Vertex-Finding Strategies for High-Dimensional Problems}\label{sec:vertex-strategies}

The outer approximation algorithm requires enumerating all vertices of $A_k$ at each iteration to select the farthest vertex $v_k$. This vertex enumeration has complexity $O(m^{\lfloor q/2 \rfloor})$ where $m$ is the number of halfspace constraints, which grows prohibitively for $q \geq 4$. We develop two approximate strategies that bypass full enumeration.

The LP probe strategy finds candidate vertices by solving $\max\{w_i^\top y : By \leq b\}$---where $By\le b$ is the halfspace representation of the current outer approximation $A_k$---for a set of probe directions $\mathcal{W}$, initialized with $\{\pm e_j, \pm \bar{w}\}$ and augmented by each normalized combined cut-normal $u_{k+1}=g_{k+1}/\|g_{k+1}\|_2$. This avoids enumeration entirely but may miss vertices that are not maximizers in any probe direction. The hybrid strategy combines LP probes with periodic full enumeration: a full-enumeration pass occurs at the first evaluation ($j=1$) and whenever $j$ is a multiple of $N=50$, while the remaining evaluations use LP probes only. This trades efficiency against omissions: LP-probe evaluations are fast, while the periodic full enumerations can reveal previously missed vertices.

Convergence is preserved under mild conditions: whenever $E_k > 0$, at least one vertex $v \in \operatorname{ext}(A_k) \setminus A$ must be found, and the selected vertex must satisfy $d_2(v_k, A) \geq c \cdot E_k$ for some $c > 0$ (for full enumeration, \rev{$c = \theta$} as in the proof of Theorem~\ref{thm:rate}; for the probe-based strategies \rev{we do not establish such a positive fraction $c$}). \rev{The packing argument of Theorem~\ref{thm:rate} uses only this fixed-fraction property; for full enumeration it therefore applies verbatim and the improved rate $O(k^{2/(1-q)})$ is retained, with a constant depending on $c$. For the implemented LP-probe and hybrid strategies the fixed-fraction property is not verified, so no convergence or rate guarantee is claimed for them.}

Figure~\ref{fig:q4-comparison} compares all four combinations of norm and strategy on Example~1 with $q=4$.
\rev{Because full enumeration is bypassed, the certified error is not available at every evaluation; the reported quantity is the \emph{candidate-set residual}---the farthest distance to $A$ over the candidate vertices currently known---an uncertified surrogate that may omit vertices the probes have missed. All four combinations terminate in 108--112 evaluations at the candidate-residual tolerance $\varepsilon = 0.0496$ (Table~\ref{tab:strategies}); the choice of norm and of vertex-finding strategy makes little difference to the evaluation count. The runs with periodic full enumeration show characteristic upward steps when previously undetected vertices are discovered.}
\begin{figure}[htbp]
\centering
\includegraphics[width=0.8\textwidth]{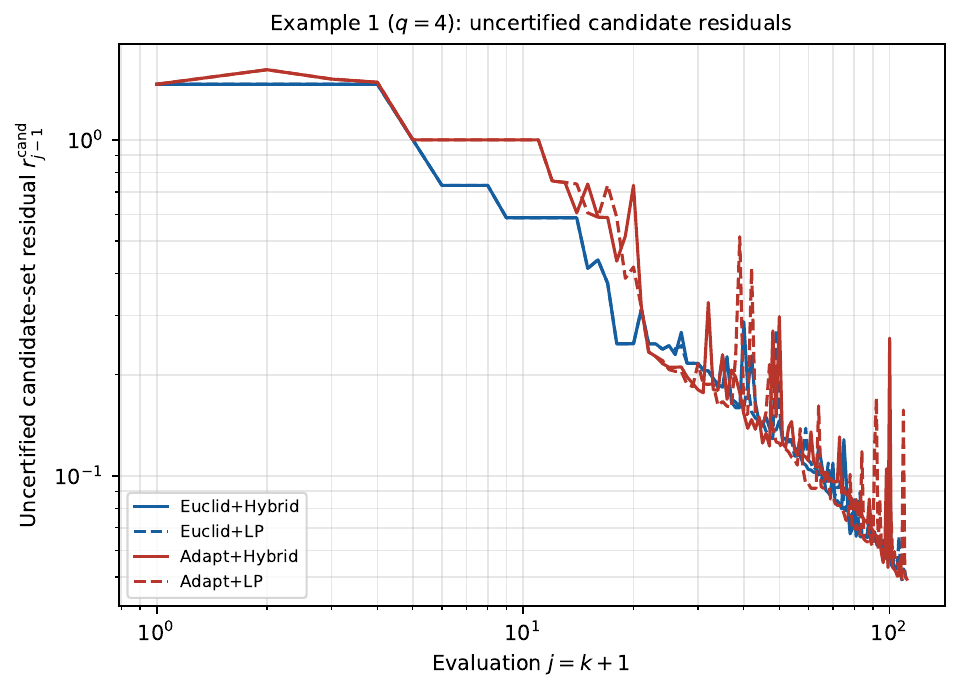}
\caption{Vertex-finding strategy comparison for Example~1 with $q=4$. All four norm--strategy combinations terminate in 108--112 evaluations at the candidate-residual tolerance $\varepsilon = 0.0496$; the plotted quantity is the uncertified candidate-set residual, against evaluation $j=k+1$.}
\label{fig:q4-comparison}
\end{figure}

\begin{table}[htbp]
\centering
\caption{Strategy comparison for Example~1 ($q=4$, $\varepsilon = 0.0496$). The fitted slope is for the \emph{uncertified candidate-set residual} (Section~\ref{sec:vertex-strategies}), not a certified Hausdorff error; $2/(1-q) = -0.67$ is shown only as a conditional reference slope, since the probe-based strategies are not proved to preserve the rate.}
\label{tab:strategies}
\begin{tabular}{lcccc}
\toprule
Configuration & Norm & Strategy & Evaluations & Fitted slope \\
\midrule
Euclid+Hybrid & Euclidean & Hybrid   & 108 & $-0.89$ \\
Euclid+LP     & Euclidean & LP probe & 111 & $-0.90$ \\
Adapt+Hybrid  & Adaptive  & Hybrid   & 108 & $-0.95$ \\
Adapt+LP      & Adaptive  & LP probe & 112 & $-0.96$ \\
\bottomrule
\end{tabular}
\end{table}

\subsection{Practical Recommendations}

Based on the experiments, we offer the following comments. \rev{The adaptive metric should not be chosen for speed: on these instances it does not reduce the evaluation count relative to the fixed Euclidean norm (Table~\ref{tab:convergence}), and our exploratory single-run timings showed no wall-clock advantage. It is the appropriate choice when the improved-rate guarantee of Theorem~\ref{thm:rate} under a problem-adapted scalarization metric is wanted---that is, when a geometry-aware scalarization is desired while retaining the fixed-Euclidean convergence guarantee; the $\varepsilon_0$-regularization then keeps the metric uniformly well-conditioned by construction.}

For vertex-finding, full enumeration is tractable and should be preferred when $q \leq 3$. For $q \geq 4$, where full enumeration becomes expensive, the hybrid and LP-probe strategies are practical heuristics rather than guaranteed methods; on our single $q=4$ instance the hybrid strategy with periodic enumeration (a full pass at the first evaluation and every $50$th evaluation thereafter) traded speed against vertex detection reasonably, while pure LP probing was faster but may miss vertices. We did not establish a convergence guarantee for either (Section~\ref{sec:vertex-strategies}), so these are empirical observations on one test.

The parameter $\theta_k$ serves as a useful real-time diagnostic of metric conditioning. \rev{Across the four main runs the terminal $\theta_k$ lay in }\rev{$[0.46, 0.67]$}\rev{, and $\theta_k$ never fell below its floor $\sqrt{\varepsilon_0/(\varepsilon_0+1)}\approx 0.30$}. If $\theta_k$ \rev{approaches this floor}, a larger regularization parameter $\varepsilon_0$ \rev{raises it} at the cost of slower adaptation. Our default $\varepsilon_0 = 0.1$ performed well across all problems\rev{, and the sensitivity study of Section~\ref{sec:eps0-sensitivity} (Table~\ref{tab:eps0-sensitivity}) supports this choice empirically}.


\section{Conclusion}\label{sec:conclusion}

\begin{revblock}
We have established that the improved $O(k^{2/(1-q)})$ convergence rate for norm-minimization--based outer approximation extends from the Euclidean norm to all inner-product norms via a linear isometry and---our main result---that the \emph{adaptive-metric} algorithm attains the same improved rate. A Tikhonov-type $\varepsilon_0$-regularization keeps the iteration-dependent metric uniformly well-conditioned, and an adaptation of the Euclidean packing argument to the moving metric transfers the improved exponent to the adaptive setting, with no strict-convexity, smoothness, or dispersion hypothesis on the image set (under the standing assumptions of Theorem~\ref{thm:rate}). The conditioning of the metric enters only the multiplicative constant, and the regularization is what the proof uses: the packing argument requires a uniform bound on $\kappa(M_k)$, which the $\varepsilon_0$-regularization supplies; whether the improved exponent can persist without it we leave open. Numerical experiments illustrate behaviour consistent with the theoretical rate for both the fixed Euclidean and the adaptive metric---the fitted slopes are consistent with the predicted exponents, at comparable evaluation counts---so the value of the adaptive metric lies in its rate guarantee and uniform conditioning rather than in fewer evaluations. Conceptually, our results clarify the geometric nature of the improved exponent $2/(1-q)$, which is not tied to the standard Euclidean structure but to the Hilbertian geometry of the norm; the role of the specific metric is confined to the constant---the condition number $\kappa(M)$ and the circumradius of the transformed slice $\tilde A$ for a fixed norm, and the regularized conditioning $\theta^{-1}=\sqrt{(\varepsilon_0+1)/\varepsilon_0}$ for the adaptive variant. Finally, the empirical second-moment matrix of the cut-normals had a positive smallest eigenvalue at termination in the four main $q\le3$ runs; Remark~\ref{rem:dispersion} records this observation and its diagnostic role (we do not claim a uniform-in-$k$ floor), leaving open whether optimal-rate convergence \emph{forces} such spreading.
\end{revblock}

\subsection{Open Questions}

Several directions for further research remain open. \rev{Having established the improved rate for the adaptive metric (Theorem~\ref{thm:rate}), a natural next question is how to reduce its constant $\theta^{-1}\bar\lambda(A)$---for instance through data-driven choices of the regularization $\varepsilon_0$ or of the metric update rule---so as to bring the onset of the asymptotic regime forward in practice.}

More broadly, it would be natural to extend the analysis beyond inner-product norms. The uniform convexity of the $\ell_p$-norms for $1 < p < \infty$ suggests that exponents strictly larger than $1/(1-q)$ might be attainable, with constants depending on $p$ through the modulus of convexity.

Another direction concerns the \rev{empirical dispersion of the cut-normals (Remark~\ref{rem:dispersion}). Our} experiments \rev{record $\lambda_{\min}(\Sigma_k)$ in roughly $[0.11, 0.24]$ in the four main $q\le3$ runs---an empirical observation only}. A deeper analysis of the equilibrium distribution of cut-normals could yield explicit formulas in terms of problem geometry.

Finally, \rev{the reliance on exact scalarization solves and full vertex enumeration limits the scope of the current theory. Establishing} precise convergence rates for the approximate vertex-finding strategies of Section~\ref{sec:vertex-strategies}\rev{, and for inexact scalarizations,} would broaden the practical applicability of the framework.

\section*{Funding}
This research received no external funding.

\section*{Data and Code Availability}
No external data were used in this study. All numerical results were generated by the algorithms described in the paper. The code used to produce the numerical experiments is available from the author upon reasonable request.

\section*{Acknowledgements}

The author acknowledges the support of King Fahd University of Petroleum \& Minerals (KFUPM) and the Interdisciplinary Research Center for Smart Mobility and Logistics at KFUPM. 
\bibliographystyle{spmpsci}
\bibliography{references}

@article{bensolve,
	author = {L{\"o}hne, Andreas and Wei{\ss}ing, Benjamin},
	title = {The {Vector} {Linear} {Program} {Solver} {Bensolve} -- {Notes} on {Theoretical} {Background}},
	journal = {European Journal of Operational Research},
	volume = {260},
	number = {3},
	pages = {807--813},
	year = {2017},
	doi = {10.1016/j.ejor.2016.02.039},
}

@article{desantis2020solving,
	author = {De Santis, Maria and Eichfelder, Gabriele and Niebling, Jan and Rockt{\"a}schel, Stefanie},
	title = {Solving {Multiobjective} {Mixed} {Integer} {Convex} {Optimization} {Problems}},
	journal = {SIAM Journal on Optimization},
	volume = {30},
	number = {4},
	pages = {3122--3145},
	year = {2020},
	doi = {10.1137/19M1264709},
}

@article{klamroth2007constrained,
	author = {Klamroth, Kathrin and Tind, J{\"o}rgen},
	title = {Constrained {Optimization} {Using} {Multiple} {Objective} {Programming}},
	journal = {Journal of Global Optimization},
	volume = {37},
	number = {3},
	pages = {325--355},
	year = {2007},
	doi = {10.1007/s10898-006-9052-x},
}

@article{klamroth2003unbiased,
	author = {Klamroth, Kathrin and Tind, J{\"o}rgen and Wiecek, Margaret M.},
	title = {Unbiased {Approximation} in {Multicriteria} {Optimization}},
	journal = {Mathematical Methods of Operations Research},
	volume = {56},
	number = {3},
	pages = {413--437},
	year = {2003},
	doi = {10.1007/s001860200217},
}

@article{niebling2019branch,
	author = {Niebling, Jan and Eichfelder, Gabriele},
	title = {A {Branch}-and-{Bound}-{Based} {Algorithm} for {Nonconvex} {Multiobjective} {Optimization}},
	journal = {SIAM Journal on Optimization},
	volume = {29},
	number = {1},
	pages = {794--821},
	year = {2019},
	doi = {10.1137/18M1169680},
}

@article{wagner2023algorithms,
	author = {Wagner, Alexander and Ulus, Firdevs and Rudloff, Birgit and Kov{\'a}{\v{c}}ov{\'a}, Gabriela and Hey, Norman},
	title = {Algorithms to {Solve} {Unbounded} {Convex} {Vector} {Optimization} {Problems}},
	journal = {SIAM Journal on Optimization},
	volume = {33},
	number = {4},
	pages = {2598--2624},
	year = {2023},
	doi = {10.1137/22M1507693},
}

@book{lotov2004interactive,
	author = {Lotov, Aleksandr V. and Bushenkov, Vladimir A. and Kamenev, Georgy K.},
	title = {Interactive {Decision} {Maps}: {Approximation} and {Visualization} of {Pareto} {Frontier}},
	series = {Applied Optimization},
	volume = {89},
	address = {Boston, MA},
	publisher = {Springer},
	year = {2004},
	doi = {10.1007/978-1-4419-8851-5},
}

@article{keskin_outer_2023,
	author = {Keskin, {\.I}rem Nur and Ulus, Firdevs},
	title = {Outer {Approximation} {Algorithms} for {Convex} {Vector} {Optimization} {Problems}},
	journal = {Optimization Methods and Software},
	volume = {38},
	number = {4},
	pages = {723--755},
	year = {2023},
	doi = {10.1080/10556788.2023.2167994},
}

@book{ansari_vector_2018,
	author = {Ansari, Qamrul Hasan and K{\"o}bis, Elisabeth and Yao, Jen-Chih},
	title = {Vector {Variational} {Inequalities} and {Vector} {Optimization}},
	series = {Vector Optimization},
	address = {Cham},
	publisher = {Springer International Publishing},
	year = {2018},
	doi = {10.1007/978-3-319-63049-6},
}

@article{eichfelder_local_2026,
	author = {Eichfelder, Gabriele and Ulus, Firdevs},
	title = {Local {Upper} {Bounds} {Based} on {Polyhedral} {Ordering} {Cones}},
	journal = {EURO Journal on Computational Optimization},
	volume = {14},
	pages = {100124},
	year = {2026},
	doi = {10.1016/j.ejco.2025.100124},
}

@article{eichfelder_approximation_2022,
	author = {Eichfelder, Gabriele and Warnow, Leo},
	title = {An {Approximation} {Algorithm} for {Multi}-{Objective} {Optimization} {Problems} {Using} a {Box}-{Coverage}},
	journal = {Journal of Global Optimization},
	volume = {83},
	number = {2},
	pages = {329--357},
	year = {2022},
	doi = {10.1007/s10898-021-01109-9},
}

@misc{cvx,
	author = {Grant, Michael and Boyd, Stephen},
	title = {{CVX}: {MATLAB} {Software} for {Disciplined} {Convex} {Programming}, {Version} 2.2},
	howpublished = {\url{http://cvxr.com/cvx}},
	year = {2020},
}

@article{nobakhtian2017benson,
	author = {Nobakhtian, Saeed and Shafiei, Negin},
	title = {A {Benson} {Type} {Algorithm} for {Nonconvex} {Multiobjective} {Programming} {Problems}},
	journal = {TOP},
	volume = {25},
	number = {2},
	pages = {271--287},
	year = {2017},
	doi = {10.1007/s11750-016-0430-3},
}

@article{lohne_primal_2014,
	author = {L{\"o}hne, Andreas and Rudloff, Birgit and Ulus, Firdevs},
	title = {Primal and {Dual} {Approximation} {Algorithms} for {Convex} {Vector} {Optimization} {Problems}},
	journal = {Journal of Global Optimization},
	volume = {60},
	number = {4},
	pages = {713--736},
	year = {2014},
	doi = {10.1007/s10898-013-0136-0},
}

@article{ehrgott_approximation_2011,
	author = {Ehrgott, Matthias and Shao, Lizhen and Sch{\"o}bel, Anita},
	title = {An {Approximation} {Algorithm} for {Convex} {Multi}-{Objective} {Programming} {Problems}},
	journal = {Journal of Global Optimization},
	volume = {50},
	number = {3},
	pages = {397--416},
	year = {2011},
	doi = {10.1007/s10898-010-9588-7},
}

@incollection{gb08,
	author = {Grant, Michael and Boyd, Stephen},
	title = {Graph {Implementations} for {Nonsmooth} {Convex} {Programs}},
	booktitle = {Recent {Advances} in {Learning} and {Control}},
	editor = {Blondel, Vincent D. and Boyd, Stephen P. and Kimura, Hidenori},
	series = {Lecture Notes in Control and Information Sciences},
	publisher = {Springer},
	address = {London},
	pages = {95--110},
	year = {2008},
	doi = {10.1007/978-1-84800-155-8_7},
}

@article{kamenev2002conjugate,
	author = {Kamenev, G. K.},
	title = {Conjugate {Adaptive} {Algorithms} for {Polyhedral} {Approximation} of {Convex} {Bodies}},
	journal = {Computational Mathematics and Mathematical Physics},
	volume = {42},
	number = {9},
	pages = {1301--1316},
	year = {2002},
}

@article{kamenev1992class,
	author = {Kamenev, G. K.},
	title = {A {Class} of {Adaptive} {Algorithms} for {Approximating} {Convex} {Bodies} by {Polyhedra}},
	journal = {Computational Mathematics and Mathematical Physics},
	volume = {32},
	number = {1},
	pages = {114--127},
	year = {1992},
}

@article{gruber_approximation_1982,
	author = {Gruber, Peter M. and Kenderov, Petar},
	title = {Approximation of {Convex} {Bodies} by {Polytopes}},
	journal = {Rendiconti del Circolo Matematico di Palermo},
	volume = {31},
	number = {2},
	pages = {195--225},
	year = {1982},
	doi = {10.1007/BF02844354},
}

@book{lohne_vector_2011,
	author = {L{\"o}hne, Andreas},
	title = {Vector {Optimization} with {Infimum} and {Supremum}},
	series = {Vector Optimization},
	address = {Berlin, Heidelberg},
	publisher = {Springer},
	year = {2011},
	doi = {10.1007/978-3-642-18351-5},
}

@article{benson_outer_1998,
	author = {Benson, Harold P.},
	title = {An {Outer} {Approximation} {Algorithm} for {Generating} {All} {Efficient} {Extreme} {Points} in the {Outcome} {Set} of a {Multiple} {Objective} {Linear} {Programming} {Problem}},
	journal = {Journal of Global Optimization},
	volume = {13},
	number = {1},
	pages = {1--24},
	year = {1998},
	doi = {10.1023/A:1008215702611},
}

@book{jahn_vector_2011,
	author = {Jahn, Johannes},
	title = {Vector {Optimization}: {Theory}, {Applications}, and {Extensions}},
	address = {Berlin, Heidelberg},
	publisher = {Springer},
	year = {2011},
	doi = {10.1007/978-3-642-17005-8},
}

@article{ararat_norm_2022,
	author = {Ararat, {\c{C}}a{\u{g}}{\i}n and Ulus, Firdevs and Umer, Muhammad},
	title = {A {Norm} {Minimization}-{Based} {Convex} {Vector} {Optimization} {Algorithm}},
	journal = {Journal of Optimization Theory and Applications},
	volume = {194},
	number = {2},
	pages = {681--712},
	year = {2022},
	doi = {10.1007/s10957-022-02045-8},
}

@article{ararat_convergence_2024,
	author = {Ararat, {\c{C}}a{\u{g}}{\i}n and Ulus, Firdevs and Umer, Muhammad},
	title = {Convergence {Analysis} of a {Norm} {Minimization}-{Based} {Convex} {Vector} {Optimization} {Algorithm}},
	journal = {SIAM Journal on Optimization},
	volume = {34},
	number = {3},
	pages = {2700--2728},
	year = {2024},
	doi = {10.1137/23M1574580},
}

@article{Barber1996,
	title = {The quickhull algorithm for convex hulls},
	volume = {22},
	issn = {0098-3500},
	url = {https://dl.acm.org/doi/10.1145/235815.235821},
	doi = {10.1145/235815.235821},
	number = {4},
	urldate = {2026-07-13},
	journal = {ACM Transactions on Mathematical Software (TOMS)},
	author = {Barber, C. Bradford and Dobkin, David P. and Huhdanpaa, Hannu},
	month = dec,
	year = {1996},
	pages = {469--483},
}

\end{document}